\documentclass[article,11pt]{article}
\usepackage{graphicx}
\usepackage{amsmath,amsthm}
\usepackage{mathrsfs,amsfonts}
\usepackage{amsfonts}
\usepackage{amsmath,amssymb,amsfonts}
\usepackage{algorithm}
\usepackage{algpseudocode}
\usepackage{tikz}
\usepackage{amssymb}
\usepackage{color}
\usepackage{amsmath}  
\usepackage{amsfonts} 
\usepackage{graphicx} 
\usepackage{braket}
\usepackage{enumerate}
\usepackage{indentfirst}
\usepackage{longtable}
\usepackage{setspace}
\allowdisplaybreaks[4]

\def\si{{\sigma}}

\def\<{\left<}\def\>{\right>}\def\({\left(}\def\){\right)}

\newfam\msbmfam\font\tenmsbm=msbm10\textfont
\msbmfam=\tenmsbm\font\sevenmsbm=msbm7

\scriptfont\msbmfam=\sevenmsbm\def\bb#1{{\fam\msbmfam #1}}

\def\FF{\bb F}
\def\NN{\bb N}\def\PP{\bb P}
\def\RR{\bb R}\def\SS{\bb S}
\def\cB{{\cal B}}\def\cL{{\cal L}}
\def\cF{{\cal F}}
\def\cH{{\cal H}}
\def\cL{{\cal L}}
\def\cP{{\cal P}}\def\cS{{\cal S}}\def\cU{{\cal U}}

\usepackage{amsmath,amsfonts,amsthm,amssymb}    
\usepackage{gensymb, mathrsfs}                  

\usepackage{accents}
\DeclareMathSymbol{\widehatsym}{\mathord}{largesymbols}{"62}

\def\*#1{\mathbf{#1}}

\renewcommand{\bar}{\overline}
\renewcommand{\tilde}{\widetilde}
\renewcommand{\phi}{\varphi}

\makeatletter \renewcommand\d[1]{\ensuremath{%
  \;\mathrm{d}#1\@ifnextchar\d{\!}{}}}
\makeatother

\theoremstyle{plain}
\newtheorem{thm}{Theorem}[section]

\theoremstyle{definition}

\newtheorem{assume}[thm]{Assumption}


\newcommand{\beq}{\begin{equation}}
\newcommand{\eeq}{\end{equation}}

\definecolor{c}{rgb}{0.9,0.3,0.1}
\definecolor{b}{rgb}{0.1,0.3,0.9}

\newtheorem{remark}{Remark}[section]
\newtheorem{lemma}{Lemma}[section]
\newtheorem{proposition}{Proposition}[section]
\newtheorem{theorem}{Theorem}[section]
\newtheorem{definition}{Definition}[section]

\renewcommand{\theequation}{\arabic{section}.\arabic{equation}}

\def\si{{\sigma}}

\def\<{\left<}\def\>{\right>}\def\({\left(}\def\){\right)}

\newfam\msbmfam\font\tenmsbm=msbm10\textfont
\msbmfam=\tenmsbm\font\sevenmsbm=msbm7
\scriptfont\msbmfam=\sevenmsbm\def\bb#1{{\fam\msbmfam #1}}

\def\NN{\bb N}\def\PP{\bb P}
\def\RR{\bb R}
\def\FF{\bb F}

\def\cB{{\cal B}}\def\cF{{\cal F}}
\def\cH{{\cal H}}\def\cK{{\cal K}}\def\cL{{\cal L}}\def\cP{{\cal P}}
\def \cS{{\cal S}}

\numberwithin{equation}{section}

\begin{document}

\title{\large \bf Constrained Zero-Sum Stochastic Linear-Quadratic Differential Game for Jump-Diffusion Systems with Random Coefficients}

\author{ Yanyan Tang\thanks{Department of Applied Mathematics,  Hong Kong Polytechnic University, Hong Kong, China ({\tt 12131233@mail.sustech.edu.cn}).}  , \  
Xun Li\thanks{ Department of Applied Mathematics,  Hong Kong Polytechnic University, Hong Kong, China ({\tt li.xun@polyu.edu.hk}).}, \  
Jie Xiong\thanks{ Department of Mathematics and SUSTech International Center for Mathematics, Southern University of Science and Technology, Shenzhen, Guangdong, 518055, China ({\tt xiongj@sustech.edu.cn})
This author is supported by  National Key R\&D Program of China grant 2022YFA1006102, and
National Natural Science Foundation of China Grants 11831010 and 12326368.}\; }
\date{}
\maketitle 
 \bigskip

  \textbf{Abstract.} 
This paper studies a two-player zero-sum stochastic
linear-quadratic (SLQ) differential game for controlled
jump-diffusion systems with random coefficients, where the controls of both players are constrained to nonempty closed convex cones. Under a uniform convexity--concavity condition, we establish the existence and uniqueness of an open-loop saddle point and characterize it by a forward--backward stochastic differential equation with jumps (FBSDEJ) together with cone-type variational inequalities.  Assuming the existence of positive bounded solutions to the associated system of indefinite extended stochastic Riccati equations with jumps (IESREJs), we derive a feedback-form representation of the unique open-loop saddle point by constructing predictable minimax selectors and combining the Meyer--It\^o formula with jumps, and the FBSDEJ characterization. Finally, under additional structural conditions, we prove the existence of positive bounded solutions to the IESREJs by a double-truncation approximation and a multidimensional BSDEJ comparison theorem.

\bigskip

\section{Introduction}
Let $(\Omega,\mathcal F,\mathbb P)$ be a complete probability space supporting a one-dimensional Brownian motion $W$ and a Poisson random measure $N(dt,dz)$ on $[0,T]\times U_0$, where $U_0$ is a nonempty Borel subset of a Euclidean space. We assume that $W$ and $N$ are independent and that the compensator of $N$ is $\nu(dz)dt$, with $\nu(U_0)<\infty$. Let $\mathbb F=\{\mathcal F_t\}_{0\le t\le T}$ be the usual augmentation of the filtration generated by $W$ and $N$, and write $
\widetilde N(dt,dz):=N(dt,dz)-\nu(dz)dt.$

We consider the scalar controlled jump-diffusion
\begin{equation}\label{state}
\left\{
\begin{aligned}
dX(t)={}&\Big(A(t)X(t)+B_1(t)^\top u_1(t)+B_2(t)^\top u_2(t)\Big)dt\\
&+\Big(C(t)X(t)+D_1(t)^\top u_1(t)+D_2(t)^\top u_2(t)\Big)dW(t)\\
&+\int_{U_0}\Big(E(t,z)X(t-)+F_1(t,z)^\top u_1(t)
+F_2(t,z)^\top u_2(t)\Big)\widetilde N(dt,dz),\\
X(0)={}&x,
\end{aligned}
\right.
\end{equation}
where $x\in\mathbb R$. Let $\mathcal P$ denote the $\mathbb F$-predictable $\sigma$-field and $\cB(U_0)$ denote the Borel $\si$-field on $U_0$. The processes $A,B_1,B_2,C,D_1,D_2$ are $\mathbb F$-predictable, whereas $E,F_1,F_2$ are $\mathcal P\otimes\mathcal B(U_0)$-measurable.

Let $\Pi_k\subseteq\mathbb R^{m_k}$, $k=1,2$, be nonempty closed convex cones. Thus, if $u\in\Pi_k$, then $\lambda u\in\Pi_k$ for all $\lambda\ge0$.  The processes $u_1(\cdot)$ and $u_2(\cdot)$ denote the controls of Player 1 and Player 2, respectively, and they belong to the following spaces:
\[
\mathcal U_k
:=
\left\{
 u_k:[0,T]\times\Omega\to\Pi_k:
 u_k\text{ is predictable and }
 \mathbb E\int_0^T|u_k(t)|^2dt<\infty
\right\}.
\]
Set $\mathcal U:=\mathcal U_1\times\mathcal U_2$, and denote by $X_x^{u_1,u_2}$ the state process corresponding to $(u_1,u_2)\in\mathcal U$. The performance functional is
\begin{equation}\label{cost}
\begin{aligned}
J(x;u_1,u_2)
:=\mathbb E\Bigg[&\int_0^T\Big(
Q\big|X_x^{u_1,u_2}(t)\big|^2
+u_1(t)^\top R_{11}(t)u_1(t)
+2u_1(t)^\top R_{12}(t)u_2(t)
\\
&\ \  +u_2(t)^\top R_{22}(t)u_2(t)
\Big)dt+G(T)\big|X_x^{u_1,u_2}(T)\big|^2\Bigg].
\end{aligned}
\end{equation}
Here $Q,R_{11},R_{12},R_{22}$ are predictable and $G(T)$ is $\mathcal F_T$-measurable. 

In our problem, the functional (\ref{cost}) can be regarded as the loss of Player 1 and the gain of Player 2. Therefore, Player 1 aims to minimize (\ref{cost}) by selecting a control $u_1\in\cU_1$, while Player 2 aims to maximize (\ref{cost}) by selecting a control $u_2\in\cU_2$. Then the problem considered in the paper is
the constrained two-person zero-sum SLQ differential game of jump diffusion systems with random coefficients (denoted as Problem (C-ZLQJ)), which is specifically described as follows:

\medskip
\noindent\textbf{Problem (C-ZLQJ).} For any $x\in\RR$,  find an admissible control pair  $(u_1^*, u_2^*)\in \cU_1\times\cU_2$ such that 
$$\inf_{u_1\in\cU_1}J(x; u_1, u_2^*)=J(x; u_1^*,u_2^*)=\sup_{u_2\in\cU_2}J(x; u_1^*,u_2).$$

The study of zero-sum differential games can be traced back to the pioneering work of Fleming and   Souganidis \cite{MR956957}, which  proved that upper and lower value functions satisfy the dynamic programming principle  and  are the unique viscosity solutions to the associated Hamilton-Jacobi-Bellman-Isaacs (HJBI) equations. Subsequent BSDE approaches were developed by Buckdahn and Li \cite{MR2373477} for discussed  two-player  zero-sum stochastic differential games   with recursive utilities.  Based on BSDE theory, Wang and Yu \cite{MR2675843,MR2889426}  studied non-zero-sum differential games where the state process is described by BSDEs. Mou and Yong \cite{MR499112900323} studied open-loop saddle points by a Hilbert-space method, whereas Sun and Yong
\cite{MR3342161} developed both open-loop and closed-loop solvability theories.  Sun \cite{MR4253799} analyzed open-loop saddle points and the open-loop lower and upper values, and showed that under a sufficient condition the associated Riccati equation
admits a unique strongly regular solution, yielding a closed-loop representation of the open-loop saddle point.   Wu et al. \cite {wu2024open} extended the work of \cite {MR4253799} and explored SLQ differential games with regime switching. For more related works, see \cite{{MR4569666}, {MR4886360}, {MR33421611301}}.

 To place the random-coefficient game in context, we first
recall the corresponding Riccati theory for one-player SLQ
control. For random-coefficient SLQ control problems,
feedback synthesis is closely tied to the solvability of the
associated stochastic Riccati equations (SREs), which are
matrix-valued BSDEs with highly nonlinear generators. Early
work goes back to Bismut \cite{BI1976}, who initiated the
study of random-coefficient SLQ control and treated several
special cases. Tang \cite{T2003} subsequently established
the existence and uniqueness of solutions to the general
matrix-valued SRE by relating it to the stochastic Hamiltonian system and the associated stochastic flows. For jump-diffusion systems,  Zhang, Dong, and Meng
\cite{ZDM2020} studied matrix-valued stochastic Riccati equations with jumps (SREJs) associated with
random-coefficient SLQ control problems driven jointly by Brownian and Poisson noises, and proved the existence and uniqueness of their solutions. 
 
The Riccati equations corresponding to zero-sum SLQ differential games with random coefficients are highly nonlinear and indefinite SREs. Indefinite Riccati equations also arise in SLQ problems and there have been some results concerning their solvability in some special cases, see Hu and Zhou \cite{MR2175763003}, Du \cite{MR237347712003}, Sun et al. \cite{MR4254486}, and  Qian and Zhou \cite{MR4991129003}.  In contrast, the indefinite SREs derived from zero-sum SLQ differential games are fundamentally different from those obtained in optimal control problems. The key reason lies in the conflicting objectives of the two players in a zero-sum game, which naturally result in control weighting matrices with opposite signs and, consequently, an indefinite formulation. Hence, proving the solvability of SREs derived from zero-sum SLQ differential games presents substantial analytical difficulties. Due to these analytical difficulties, relatively few results are available for zero-sum SLQ problems with random coefficients. A particularly close work is that of Zhang and
Xu \cite{MR4837915}, who studied zero-sum SLQ differential games with random coefficients and non-Markovian regime switching and proved the solvability of a class of multidimensional indefinite SREs using BSDE methods.

 In mathematical finance, no-shorting restrictions in  mean--variance portfolio selection provide a canonical
motivation for cone constraints; see \cite{MR1882807}. In that setting, the admissible control set is the nonnegative
orthant, while more general models allow the control to take values in an arbitrary nonempty closed convex cone. Hu and
Zhou \cite{MR2175763} established the solvability of the associated ESREs using BSDE theory and truncation techniques. Dong \cite{MR237347712} incorporated a random
jump, leading to a system of ESREs with jumps (ESREJs), whose solvability was obtained through two recursive BSDE
systems. Hu, Shi, and Xu \cite{MR4386532,HSX2025} subsequently treated
regime-switching and Wiener--Poisson models, respectively, using multidimensional BSDE and BSDEJ comparison methods.
More recently, Shi and Xu \cite{MR4991129} studied a cone-constrained SLQ control problem with random coefficients under regime switching and controlled jump size, and established the solvability of the resulting
multidimensional fully coupled SREJs. These works concern one-player control problems with nonnegative or positive-semidefinite weighting structures. They do not
directly cover the indefinite minimax Hamiltonians arising in the present zero-sum game. In particular, the opposing
objectives of the two players lead naturally to indefinite control weights and hence to indefinite extended SREs with
jumps (IESREJs), so the preceding ESRE solvability results are not directly applicable.

Our work extends the controlled stochastic problem proposed by Hu et al. \cite{HSX2025} to the framework of zero-sum stochastic differential games. Relative to \cite{HSX2025}, the present model replaces a one-player minimization problem by a two-player zero-sum game, so one must establish coercivity in the minimizing control and anti-coercivity in the maximizing control, construct predictable saddle selectors, and verify an open-loop saddle point. The most closely related study is the work of Zhang and Xu \cite{MR4837915}. Relative to \cite{MR4837915}, our dynamics contain a controlled Poisson jump term rather than regime switching alone. The
generators of the associated indefinite extended stochastic Riccati equations with jumps depend jointly on the two pre-jump Riccati variables and on the post-jump quantities
$
P_k(t-)+\Gamma_k(t,z), t\in[0,T],\  k=1,2.
$
Jumps may move the scalar state across zero, so the positive and negative state regions cannot be treated independently. Moreover, while \cite{MR4837915} derives closed-loop feedback control-strategy pairs, our objective is to identify the unique open-loop saddle point and then prove that it admits a feedback-form representation.

The main contributions of this paper are as follows. First, under a uniform convexity-concavity condition (UCC), we establish unique open-loop solvability and characterize the saddle point by a FBSDEJ together with cone-type variational inequalities. Second, assuming that the associated indefinite extended stochastic Riccati equations with jumps admit positive bounded solutions, we construct predictable minimax selectors, prove well-posedness and admissibility of the resulting closed-loop state-control pair, and verify that the feedback-form candidate coincides with the unique open-loop saddle point. The value is then expressed in terms of the initial Riccati variables. Third, under additional structural conditions, we prove the existence of positive bounded solutions to the coupled Riccati system. The proof combines a double truncation, a  BSDEJ comparison theorem, uniform pathwise and post-jump barriers, and an exponential stability argument for the two successive monotone limits.

The remainder of the paper is organized as follows. Section 2 introduces the notation and standing assumptions. Section 3 establishes open-loop solvability and the FBSDEJ characterization. Section 4 derives and verifies the feedback representation. Section 5 proves the solvability of the associated indefinite extended stochastic Riccati equations with jumps.

\section{Preliminaries}\label{sec1}

For $x\in\mathbb R$, set $
x^+:=\max\{x,0\}, x^-:=\max\{-x,0\}.
$
The superscript $\top$ denotes transpose, and $|\cdot|$
denotes the Euclidean norm for vectors and the Frobenius norm for matrices. Let $\mathbb S^n$ be the space of real symmetric $n\times n$ matrices and let $\mathbb S_+^n$ be the positive semidefinite matrices.  The identity matrix of order $n$ is denoted by $I_n$.

Let $\mathbb H$ be a finite-dimensional Euclidean space. We
use the following spaces:
\begin{align*}
L_{\mathcal F_T}^{\infty}(\Omega;\mathbb H)
:={}&
\left\{
\xi:\Omega\to\mathbb H:
\xi\text{ is }\mathcal F_T\text{-measurable and }
\operatorname*{ess\,sup}_{\omega\in\Omega}|\xi(\omega)|
<\infty
\right\},
\\
L_{\mathbb F}^{2}(0,T;\mathbb H)
:={}&
\left\{
\phi:
\phi\text{ is predictable and }
\mathbb E\int_0^T|\phi(t)|^2dt<\infty
\right\},
\\
L_{\mathbb F}^{\infty}(0,T;\mathbb H)
:={}&
\left\{
\phi:
\phi\text{ is predictable and }
\operatorname*{ess\,sup}_{(\omega,t)}
|\phi(\omega,t)|<\infty
\right\},
\\
\mathcal L_{\mathcal P}^{2,\nu}(0,T;\mathbb H)
:={}&
\left\{
\phi:
\phi\text{ is }
\mathcal P\otimes\mathcal B(U_0)\text{-measurable and }
\mathbb E\int_0^T\int_{U_0}
|\phi(t,z)|^2\nu(dz)dt<\infty
\right\},
\\
\mathcal L_{\mathcal P}^{\infty,\nu}(0,T;\mathbb H)
:={}&
\left\{
\phi:
\phi\text{ is }
\mathcal P\otimes\mathcal B(U_0)\text{-measurable and }
\operatorname*{ess\,sup}_{(\omega,t,z)}
|\phi(\omega,t,z)|<\infty
\right\},
\\
\mathcal S_{\mathbb F}^{2}(0,T;\mathbb H)
:={}&
\left\{
\phi:
\phi\text{ is càdlàg and adapted, and }
\mathbb E\left[
\sup_{0\le t\le T}|\phi(t)|^2
\right]<\infty
\right\}.
\end{align*}
The essential suprema in
$L_{\mathbb F}^{\infty}$ and
$\mathcal L_{\mathcal P}^{\infty,\nu}$ are taken with respect
to $d\mathbb P\otimes dt$ and
$d\mathbb P\otimes dt\otimes\nu(dz)$, respectively.

We impose the following standing boundedness conditions.

\begin{assume}\label{A1}
\begin{enumerate}
    \item [(A1)] $A(\cdot), C(\cdot) \in L_{\FF}^\infty(0,T; \RR), B_k(\cdot), D_k (\cdot)\in L_{\FF}^\infty(0,T; \RR^{m_k}) $, $E(\cdot, \cdot)\in\cL_{\cP}^{\infty, \nu}(0,T; \RR)$ and $F_k(\cdot, \cdot)\in\cL_{\cP}^{\infty,\nu}(0,T; \RR^{m_k}), k=1,2$. 
    \item [(A2)] $Q(\cdot) \in L_{\FF}^\infty(0,T; \RR)$, $ R_{kk}(\cdot)\in L_{\FF}^\infty(0,T; \SS^{m_k})$,  $R_{12}(\cdot)=R_{21}^\top(\cdot) \in L_{\FF}^\infty(0,T; \RR^{m_1\times m_2})$ and  $G(T)\in L_{\cF_{T}}^\infty(\Omega; \RR), k=1,2 $.  
    
\end{enumerate}
\end{assume}
Under Assumption (A1), for any initial value  $x\in\RR$ and  control pair $(u_1, u_2)\in L_{\mathbb{F}}^2(0,T;\RR^{m_1+m_2})$ the state equation (\ref{state}) admits  a unique solution  $X\in  S^2_{\FF}(0,T; \RR)$ (See Meng \cite{Meng2014}). Furthermore,  under Assumption (A2)  the performance functional  (\ref{cost}) is well-defined. When no confusion can arise, the arguments $(\omega,t)$ are suppressed; for example, we write $A=A(t)$ and
$E(z)=E(t,z)$. Left limits are retained in the definitions
of stochastic equations and may be omitted only in
expressions integrated with respect to $dt$.

\section{Open-loop solvability}\label{sec2}
\setcounter{equation}{0}
\renewcommand{\theequation}{\thesection.\arabic{equation}}

This section establishes the unique open-loop solvability of Problem
(C-ZLQJ) under a uniform convexity--concavity condition and gives an
FBSDEJ characterization of the open-loop saddle point.

\begin{definition}\label{def:open-loop-saddle}
For an initial state $x\in\mathbb R$, a pair
$(u_1^*,u_2^*)\in\mathcal U_1\times\mathcal U_2$ is called an
\emph{open-loop saddle point} of Problem (C-ZLQJ) if
\[
J(x;u_1^*,u_2)
\le
J(x;u_1^*,u_2^*)
\le
J(x;u_1,u_2^*),
\qquad
(u_1,u_2)\in\mathcal U_1\times\mathcal U_2.
\]
The associated state process
$X^*=X_x^{u_1^*,u_2^*}$ is called the optimal state. The problem is
said to be uniquely open-loop solvable at $x$ if the open-loop saddle
point is unique, and uniquely open-loop solvable if this holds for every
$x\in\mathbb R$.
\end{definition}

For
$(\widetilde u_1,\widetilde u_2)
\in
L_{\mathbb F}^2(0,T;\mathbb R^{m_1+m_2})
$,
let $X_0^{\widetilde u_1,\widetilde u_2}$ denote the zero-initial-state
solution and define
\begin{align*}
\widetilde J(0;\widetilde u_1,\widetilde u_2)
:=\mathbb E\Bigg[
\int_0^T
\Big(
Q(t)|X_0^{\widetilde u_1,\widetilde u_2}|^2
+\widetilde u_1^\top R_{11}\widetilde u_1
+2\widetilde u_1^\top R_{12}\widetilde u_2
+\widetilde u_2^\top R_{22}\widetilde u_2
\Big)dt
+G(T)|X_0^{\widetilde u_1,\widetilde u_2}(T)|^2
\Bigg].
\end{align*}

\begin{assume}\label{A2}
\textnormal{(UCC)} There exists a constant $\delta>0$ such that
\begin{equation}\label{A21}
\left\{
\begin{aligned}
\widetilde J(0;\widetilde u_1,0)
&\ge
\delta\,
\mathbb E\int_0^T|\widetilde u_1(t)|^2dt,
&&
\forall \widetilde u_1\in L_{\mathbb F}^2(0,T;\mathbb R^{m_1}),
\\
\widetilde J(0;0,\widetilde u_2)
&\le
-\delta\,
\mathbb E\int_0^T|\widetilde u_2(t)|^2dt,
&&
\forall\widetilde u_2\in L_{\mathbb F}^2(0,T;\mathbb R^{m_2}).
\end{aligned}
\right.
\end{equation}
\end{assume}

\begin{theorem}\label{26081}
  Let Assumptions \ref{A1} and   \ref{A2} hold. Then Problem (C-ZLQJ) is uniquely open-loop solvable for every initial value $x\in\RR$.
\end{theorem}
\begin{proof}
Adapting  the Hilbert-space formulation in
\cite[Assumption (A3) and Theorem 4.4]{MR4253799}, set
$
\mathbb H_i
:=
L_{\mathbb F}^2(0,T;\mathbb R^{m_i}),\ 
\mathbb H:=\mathbb H_1\times\mathbb H_2,
$
Since $\Pi_i$ is a nonempty closed convex cone,
$\mathcal U_i$ is a nonempty closed convex subset of
$\mathbb H_i$; hence $\cU :=\mathcal U_1\times\mathcal U_2$ is a nonempty closed
convex subset of $\mathbb H$.

Let $u=(u_1,u_2)$. By Assumption 2.1 and the standard estimate for the linear
state equation, the map
$
u\longmapsto X_0^u
$
is bounded and linear from $\mathbb H$ into
$\mathcal S_{\mathbb F}^2(0,T;\mathbb R)$. Since
$
X_x^{u_1,u_2}=X_x^{0,0}+X_0^{u_1,u_2},
$
the control-to-state map for fixed $x$ is continuous and affine. Consequently,
$
J_x(u):=J(x;u_1,u_2)
$
is a continuously Fr\'echet differentiable quadratic
functional on $\mathbb H$. Define
$
\mathfrak F_x(u)
:=
\big(
\nabla_{u_1}J_x(u),
-\nabla_{u_2}J_x(u)
\big).
$
Then $\mathfrak F_x$ is Lipschitz continuous: for some
constant $L>0$,
\[
\|\mathfrak F_x(u)-\mathfrak F_x(v)\|_{\mathbb H}
\le
L\|u-v\|_{\mathbb H},
\qquad u,v\in\mathbb H.
\]

By polarization of the quadratic functional and the
symmetry of its associated bilinear form,
\[
\begin{aligned}
\left\langle
\mathfrak F_x(u)-\mathfrak F_x(v),
u-v
\right\rangle_{\mathbb H}
=
2\widetilde J
\big(
0;u_1-v_1,0
\big)
-
2\widetilde J
\big(
0;0,u_2-v_2
\big).
\end{aligned}
\]
Hence Assumption \ref{A2} yields
\begin{equation}\label{eq:strong-monotonicity-open-loop}
\left\langle
\mathfrak F_x(u)-\mathfrak F_x(v),
u-v
\right\rangle_{\mathbb H}
\ge
2\delta\|u-v\|_{\mathbb H}^2,
\qquad u,v\in\mathbb H.
\end{equation}
Thus $\mathfrak F_x$ is strongly monotone.

Let $\operatorname{Proj}_{\cU}$ denote the metric
projection onto $\cU$. For
$0<\rho<4\delta/L^2$, define
$
\mathcal T_\rho(u)
:=
\operatorname{Proj}_{\cU}
\big(
u-\rho\mathfrak F_x(u)
\big).
$
Using the nonexpansiveness of the metric projection,
\eqref{eq:strong-monotonicity-open-loop}, and the Lipschitz
continuity of $\mathfrak F_x$, we obtain
\[
\begin{aligned}
\|\mathcal T_\rho(u)-\mathcal T_\rho(v)\|_{\mathbb H}^2
&\le
\big(
1-4\rho\delta+\rho^2L^2
\big)
\|u-v\|_{\mathbb H}^2.
\end{aligned}
\]
Moreover, by the strong monotonicity and Lipschitz
continuity of $\mathfrak F_x$, for $u\ne v$,
\[
\begin{aligned}
2\delta\|u-v\|_{\mathbb H}^2
&\le
\left\langle
\mathfrak F_x(u)-\mathfrak F_x(v),
u-v
\right\rangle_{\mathbb H}
\le
\|\mathfrak F_x(u)-\mathfrak F_x(v)\|_{\mathbb H}
\|u-v\|_{\mathbb H}
\le
L\|u-v\|_{\mathbb H}^2.
\end{aligned}
\]
Hence $2\delta\le L$. Therefore,
$
1-4\rho\delta+\rho^2L^2
=
(1-2\rho\delta)^2
+\rho^2(L^2-4\delta^2)
\ge0,
$
whereas
$
1-4\rho\delta+\rho^2L^2<1
$
whenever $0<\rho<4\delta/L^2$. Thus,
$
q_\rho
:=
\sqrt{1-4\rho\delta+\rho^2L^2}
\in[0,1),
$
and $\mathcal T_\rho$ is a contraction with contraction constant $q_\rho$.  It admits a unique fixed point
$
u^*=(u_1^*,u_2^*)\in\cU.
$
The characterization of the metric projection implies
\begin{equation}\label{eq:open-loop-product-VI}
\left\langle
\mathfrak F_x(u^*),
v-u^*
\right\rangle_{\mathbb H}
\ge0,
\qquad v\in\cU.
\end{equation}
Equivalently,
\begin{align*}
\left\langle
\nabla_{u_1}J_x(u^*),
v_1-u_1^*
\right\rangle_{\mathbb H_1}
&\ge0,
&&v_1\in\mathcal U_1,
\\
\left\langle
\nabla_{u_2}J_x(u^*),
v_2-u_2^*
\right\rangle_{\mathbb H_2}
&\le0,
&&v_2\in\mathcal U_2.
\end{align*}

Since $J_x$ is quadratic, its exact Taylor expansion and
Assumption \ref{A2} give, for every $v_1\in\mathcal U_1$,
\[
\begin{aligned}
J(x;v_1,u_2^*)-J(x;u_1^*,u_2^*)
={}&
\left\langle
\nabla_{u_1}J_x(u^*),
v_1-u_1^*
\right\rangle_{\mathbb H_1}
+
\widetilde J
\big(
0;v_1-u_1^*,0
\big)
\ge{}
\delta\|v_1-u_1^*\|_{\mathbb H_1}^2,
\end{aligned}
\]
whereas, for every $v_2\in\mathcal U_2$,
\[
\begin{aligned}
J(x;u_1^*,v_2)-J(x;u_1^*,u_2^*)
={}&
\left\langle
\nabla_{u_2}J_x(u^*),
v_2-u_2^*
\right\rangle_{\mathbb H_2}
+
\widetilde J
\big(
0;0,v_2-u_2^*
\big)
\le{}
-\delta\|v_2-u_2^*\|_{\mathbb H_2}^2.
\end{aligned}
\]
Consequently,
\[
J(x;u_1^*,v_2)
\le
J(x;u_1^*,u_2^*)
\le
J(x;v_1,u_2^*),
\qquad
(v_1,v_2)\in\mathcal U_1\times\mathcal U_2.
\]
Thus $u^*$ is an open-loop saddle point. Moreover, every
open-loop saddle point satisfies
\eqref{eq:open-loop-product-VI}; the uniqueness of the
solution to that variational inequality therefore implies
the uniqueness of the saddle point. Since $x\in\mathbb R$
was arbitrary, Problem (C-ZLQJ) is uniquely open-loop
solvable for every initial value.
\end{proof}

For convenience,  we introduce the following notations ( $
 \forall (t,z)\in[0,T]\times U_0$)
 \begin{eqnarray}\label{Nota1}
     B&=&( B_1^\top, \  B_2^\top)^\top, \   D=( D_1^\top, \  D_2^\top)^\top, \ F(z)=( F_1^\top(z), \  F_2^\top(z))^\top,\ u= ( u_1^\top, u_2^\top)^\top\nonumber\\
       R&=&\begin{pmatrix}
R_{11} & R_{12}\\
R_{21}  & R_{22} 
\end{pmatrix}.
 \end{eqnarray}
The next theorem characterizes the open-loop saddle point by a coupled
FBSDEJ and pointwise cone-type variational inequalities.

\begin{theorem}\label{THE265}
Let Assumptions 2.1 and 3.1 hold, and let $x\in\mathbb R$.
A control pair $u^*=(u_1^*,u_2^*)\in\mathcal U_1\times\mathcal U_2$
is the open-loop saddle point of Problem (C-ZLQJ) if and only if there exists
an adapted solution
$$
(X^*,Y^*,Z^*,\cK^*)\in
\mathcal S_{\mathbb F}^2(0,T;\mathbb R)
\times\mathcal S_{\mathbb F}^2(0,T;\mathbb R)
\times L_{\mathbb F}^2(0,T;\mathbb R)
\times L_{\mathcal P}^{2,\nu}(0,T;\mathbb R)
$$
to the FBSDEJs
\begin{equation}\label{FBSDE0701}
\left\{\begin{array}{ccl}
dX^*(t)&=&
\big(A X^*+B^\top u^*\big)dt
+\big(CX^*+D^\top u^*\big)dW(t)+\int_{U_0}\big(EX^*+F^\top u^*
\big)\widetilde N(dt,dz),\\
dY^*(t)&=&
-\Big(
AY^*+CZ^*
+\int_{U_0}E(z)\cK^*(z)\nu(dz)
+QX^*
\Big)dt+Z^*(t)dW(t)\\
&&+\int_{U_0}\cK^*(t,z)\widetilde N(dt,dz),\\
X^*(0)&=&\,x,\qquad
Y^*(T)=G(T)X^*(T),
\end{array}\right.\end{equation}
such that, for $dt\otimes d\mathbb P$-a.e.$(t,\omega)$,
\begin{equation}\label{eq07011}
\Big<B_1Y^*+D_1Z^*+R_{11}u_1^*+R_{12}u_2^*
+\int_{U_0}F_1(z)\cK^*(z)\nu(dz),
u_1-u_1^*\Big>
\ge0,
\quad \forall u_1\in\Pi_1
\end{equation}
and
\begin{equation}\label{eq07012}
  \Big<B_2Y^*+D_2Z^*+R_{21}u_1^*+R_{22}u_2^*
+\int_{U_0}F_2(z)\cK^*(z)\nu(dz),
u_2-u_2^*\Big>
\le0,
\quad \forall u_2\in\Pi_2.
\end{equation}
\end{theorem}
\begin{proof}
We first prove necessity. Suppose that $u^*=(u_1^*,u_2^*)$ is the
open-loop saddle point and let $X^*=X_x^{u_1^*,u_2^*}$.
Define the Hamiltonian function 
\begin{eqnarray*}
H(t,x,y,z, k, u_1,u_2)&=&\<y, Ax+B^\top u\>+\<z, Cx+D^\top u\>+\int_{U_0}\<k, Ex+F^\top u\>\nu(dz)\nonumber\\
&&+\frac{1}{2}\big(\<Qx,x\>+\<Ru,u\>\big).
\end{eqnarray*}
 For fixed $u_2^*$, the control $u_1^*$ minimizes $J(x;u_1,u_2^*)$ over $\mathcal U_1$. Applying the class  stochastic maximum principle  (see Tang and Li \cite{TL1994}) for cone-constrained stochastic
linear-quadratic control problems with jumps gives
\begin{equation*}
 \<\frac{\partial{H}}{\partial{u_1}}(t,X^*,Y^*,Z^*, \cK^*, u_1^*,u_2^*),
u_1-u_1^*\>
\ge0,  \quad \forall u_1\in\Pi_1,\ \text{a.e.} \ t\in[0,T], \ \PP\ \text{-a.s.} 
\end{equation*}
Similarly, for fixed $u_1^*$, the control $u_2^*$ maximizes
$J(x;u_1^*,u_2)$ over $\mathcal U_2$. Therefore,
\begin{equation*}
 \<\frac{\partial{H}}{\partial{u_2}}(t,X^*,Y^*,Z^*, \cK^*, u_1^*,u_2^*),
u_2-u_2^*\>\le0,
\quad \forall u_2\in\Pi_2, \ \text{a.e.} \ t\in[0,T]. , \ \PP\ \text{-a.s.}   
\end{equation*}
This is further equivalent to  (\ref{eq07011}) and (\ref{eq07012}) being satisfied.

Conversely, suppose that there exists a solution
$(X^*,Y^*,Z^*,K^*)$ to (\ref{FBSDE0701}) satisfying (\ref{eq07011}) and (\ref{eq07012}).
For any $u_1\in\mathcal U_1$, set 
$
\delta u_1=u_1-u_1^*,
\delta X_1=X_x^{u_1,u_2^*}-X^*.
$
Then $\delta X_1=X_0^{\delta u_1,0}$. Applying  Itô's formula for $Y^*(t)\delta X_1(t)$ and the quadratic structure of $J$,
\begin{eqnarray*}
J(x;u_1,u_2^*)-J(x;u_1^*,u_2^*)
&=&2\mathbb E\int_0^T
\<
B_1Y^*+D_1Z^*+R_{11}u_1^*+R_{12}u_2^*
+\int_{U_0}F_1\cK^*\nu(dz),
\delta u_1
\> dt\\
&&+ \widetilde J(0;\delta u_1,0) 
\end{eqnarray*}
 By Assumption 3.1 and
(\ref{eq07011}), the right-hand side is nonnegative. Hence
\[
J(x;u_1^*,u_2^*)\le J(x;u_1,u_2^*),
\quad \forall u_1\in\mathcal U_1.
\]

Similarly, for any $u_2\in\mathcal U_2$, we have 
\[
J(x;u_1^*,u_2)\le J(x;u_1^*,u_2^*),
\quad \forall u_2\in\mathcal U_2.
\]
Thus $u^*$  is an open-loop saddle point. The proof is complete.
\end{proof}
\begin{remark}
When $\Pi_i=\mathbb R^{m_i}$, $i=1,2$, the variational inequalities
\eqref{eq07011}--\eqref{eq07012} reduce to the usual stationarity equations. Under cone constraints, they are normal-cone inclusions and are generally implicit. In the next section, positive bounded solutions of the associated IESREJs are used to construct a feedback-form candidate, which is then verified by Theorem \ref{THE265}.
\end{remark}

\section{Stochastic Riccati equations and open-loop saddle points}
\label{sec3}
\setcounter{equation}{0}
\renewcommand{\theequation}{\thesection.\arabic{equation}}

The purpose of this section is to derive a feedback-form representation of the
unique open-loop saddle point. Throughout the section, whenever a c\`adl\`ag
process appears in a BSDEJ generator, it is evaluated at its left limit. Time
arguments are suppressed when no confusion can arise.

For $k,k'=1,2$, set
\[
M_{kk'}(t):=R_{kk}(t)+P_{k'}(t-)D_k(t)D_k(t)^\top
\]
and
\[
N_{kk'}(t):=
P_{k'}(t-)B_k(t)^\top
+P_{k'}(t-)D_k(t)^\top C(t)
+D_k(t)^\top\Lambda_{k'}(t).
\]
When the Hamiltonians are evaluated at generic arguments, the quantities
$M_{kk'}$ and $N_{kk'}$ are understood with the corresponding values of
$P_{k'}$ and $\Lambda_{k'}$ substituted in these formulas.
Let $(P_k,\Lambda_k,\Gamma_k)_{k=1,2}$ satisfy the IESREJs
\begin{equation}\label{csre1}
\left\{
\begin{aligned}
dP_1(t)
={}&-
\Big\{
(2A(t)+C(t)^2)P_1(t-)+2C(t)\Lambda_1(t)+Q(t)
\\
&\qquad
+H_1^*
\big(
\omega,t,P_1(t-),P_2(t-),\Lambda_1(t),
\Gamma_1(t,\cdot),\Gamma_2(t,\cdot)
\big)
\Big\}dt\\
&\qquad+\Lambda_1(t)dW(t)
+\int_{U_0}\Gamma_1(t,z)\widetilde N(dt,dz),
\\
dP_2(t)
={}&-
\Big\{
(2A(t)+C(t)^2)P_2(t-)+2C(t)\Lambda_2(t)+Q(t)
\\
&\qquad
+H_2^*
\big(
\omega,t,P_1(t-),P_2(t-),\Lambda_2(t),
\Gamma_1(t,\cdot),\Gamma_2(t,\cdot)
\big)
\Big\}dt\\
&\qquad+\Lambda_2(t)dW(t)
+\int_{U_0}\Gamma_2(t,z)\widetilde N(dt,dz),
\\
P_1(T)={}&G(T),\qquad P_2(T)=G(T).
\end{aligned}
\right.
\end{equation}

Define
\begin{align*}
&\underline H_1
(\omega,t,v_1,v_2,P_1,P_2,\Lambda_1,\Gamma_1,\Gamma_2)
\notag\\
&:=v_1^\top M_{11}v_1+2N_{11}v_1
+2v_1^\top(P_1D_1D_2^\top+R_{12})v_2
\notag\\
&\quad+
\int_{U_0}
\Big[
(P_1+\Gamma_1(z))
\big(
((1+E(z)+F_1(z)^\top v_1+F_2(z)^\top v_2)^+)^2-1
\big)
\notag\\
&\hspace{32mm}
-2P_1(E(z)+F_1(z)^\top v_1+F_2(z)^\top v_2)
\notag\\
&\hspace{32mm}
+(P_2+\Gamma_2(z))
((1+E(z)+F_1(z)^\top v_1+F_2(z)^\top v_2)^-)^2
\Big]\nu(dz),
\\
&\underline H_2
(\omega,t,v_1,v_2,P_1,P_2,\Lambda_2,\Gamma_1,\Gamma_2)
\notag\\
&:=v_1^\top M_{12}v_1-2N_{12}v_1
+2v_1^\top(P_2D_1D_2^\top+R_{12})v_2
\notag\\
&\quad+
\int_{U_0}
\Big[
(P_2+\Gamma_2(z))
\big(
((-1-E(z)+F_1(z)^\top v_1+F_2(z)^\top v_2)^-)^2-1
\big)
\notag\\
&\hspace{32mm}
+2P_2(-E(z)+F_1(z)^\top v_1+F_2(z)^\top v_2)
\notag\\
&\hspace{32mm}
+(P_1+\Gamma_1(z))
((-1-E(z)+F_1(z)^\top v_1+F_2(z)^\top v_2)^+)^2
\Big]\nu(dz),
\end{align*}
and
\begin{align*}
\bar H_1(\omega,t,v_2,P_1,\Lambda_1)
&:=v_2^\top M_{21}v_2+2N_{21}v_2,
\\
\bar H_2(\omega,t,v_2,P_2,\Lambda_2)
&:=v_2^\top M_{22}v_2-2N_{22}v_2.
\end{align*}
For $k=1,2$, set $H_k:=\underline H_k+\bar H_k$ and
\[
\underline H_k^*(\omega,t, v_2,P_1,P_2,\Lambda_k,\Gamma_1,\Gamma_2)
:=\inf_{v_1\in\Pi_1}
\underline H_k(\omega,t,v_1,v_2,P_1,P_2,\Lambda_k,\Gamma_1,\Gamma_2),
\]
\begin{equation*}\label{120331}
H_k^*(\omega,t,P_1,P_2,\Lambda_k,\Gamma_1,\Gamma_2)
:=
\sup_{v_2\in\Pi_2}\inf_{v_1\in\Pi_1}
H_k(\omega,t,v_1,v_2,P_1,P_2,\Lambda_k,\Gamma_1,\Gamma_2).
\end{equation*}

\begin{remark}\label{06161}
Suppose that \eqref{csre1} admits a c\`adl\`ag solution such that
\[
\mathbb P\big(
0<P_k(t)\le K\text{ for all }t\in[0,T]
\big)=1,
\qquad k=1,2.
\]
Then
\begin{equation}\label{postjump-bounds-section4}
0<P_k(t-)+\Gamma_k(t,z)\le K,
\qquad k=1,2,
\end{equation}
$d\mathbb P\otimes dt\otimes\nu(dz)$-a.e. Indeed, if
$
A_k:=\{(\omega,t,z):P_k(t-)+\Gamma_k(t,z)\notin(0,K]\},
$
then the compensation formula and
$P_k(T_j)=P_k(T_j-)+\Gamma_k(T_j,Z_j)$ give
\[
\mathbb E\int_0^T\int_{U_0}\mathbf 1_{A_k}\nu(dz)dt
=
\mathbb E\sum_{T_j\le T}
\mathbf 1_{\{P_k(T_j)\notin(0,K]\}}=0.
\]
Since $0\le P_k(t-)\le K$, it also follows that
$|\Gamma_k(t,z)|\le K$ a.e.
\end{remark}

By enlarging a generic constant $\bar c>0$ if necessary, we may assume
\begin{equation}\label{26051}
\begin{aligned}
&|2A+C^2|\le\bar c,
\qquad
\int_{U_0}E(z)^2\nu(dz)\le\bar c,
\qquad
|B_k+CD_k|^2\le\bar c,
\\
&D_kD_k^\top+
\int_{U_0}F_k(z)F_k(z)^\top\nu(dz)
\preceq\bar c I_{m_k},
\qquad
|Q|\le\bar c,
\qquad
|G|\le\bar c,
\end{aligned}
\end{equation}
for $k=1,2$. The standard BDG--Gronwall estimate then gives, for every
stopping time $\tau\le T$ and every predictable $v_2$ satisfying
$\mathbb E\int_0^\tau|v_2(t)|^2dt<\infty$,
\begin{equation}\label{player2-state-cost-estimate}
\mathbb E\left[
\int_0^\tau Q(t)|X_x^{0,v_2}(t)|^2dt
+K|X_x^{0,v_2}(\tau)|^2
\right]
\le
K\bar c\left(
|x|^2+
\mathbb E\int_0^\tau|v_2(t)|^2dt
\right).
\end{equation}

\begin{assume}\label{A33}
There exist constants $\underline\delta>0$, $\bar\delta>0$, and
$K\ge1$ such that
\[
R_{11}\succeq\underline\delta I_{m_1},
\qquad
R_{22}\preceq-(\bar\delta+K\bar c)I_{m_2},
\qquad
Q\ge0,
\qquad
\underline\delta\le G\le\bar c\le K,
\]
and
\[
D_kD_k^\top+
\int_{U_0}F_k(z)F_k(z)^\top\nu(dz)
\succeq\underline\delta I_{m_k},
\qquad k=1,2.
\]
\end{assume}

Assumption \ref{A33} implies the UCC condition. Indeed,
$
\widetilde J(0;\widetilde u_1,0)
\ge
\underline\delta
\mathbb E\int_0^T|\widetilde u_1(t)|^2dt.
$
Moreover, by \eqref{player2-state-cost-estimate}, with $x=0$ and
$\tau=T$, we have
$ 
\widetilde J(0;0,\widetilde u_2)
\le
-\bar\delta
\mathbb E\int_0^T|\widetilde u_2(t)|^2dt.
$
Thus Assumption 3.1 holds.

Suppose now that \eqref{csre1} admits a solution satisfying the pathwise
bounds in Remark \ref{06161}. Then Assumption \ref{A33} and
\eqref{postjump-bounds-section4} imply, for $k=1,2$,
\begin{equation}\label{2301}
M_{1k}(t)
+
\int_{U_0}
\Big(
(P_1(t-)+\Gamma_1(t,z))
\wedge
(P_2(t-)+\Gamma_2(t,z))
\Big)
F_1(t,z)F_1(t,z)^\top\nu(dz)
\succeq\eta I_{m_1},
\end{equation}
and
\begin{equation}\label{230181}
M_{2k}(t)
+
\int_{U_0}
F_2(t,z)
\Big(
(P_1(t-)+\Gamma_1(t,z))
\vee
(P_2(t-)+\Gamma_2(t,z))
\Big)
F_2(t,z)^\top\nu(dz)
\preceq-\eta I_{m_2},
\end{equation}
where $\eta:=\underline\delta\wedge\bar\delta$.

\begin{proposition}\label{prop:hamiltonian-saddle}
Let Assumptions \ref{A1} and \ref{A33} hold, and suppose
that \eqref{csre1} admits a solution satisfying the
pathwise bounds in Remark \ref{06161}. Then, for
$d\mathbb P\otimes dt$-a.e. $(\omega,t)$, the Hamiltonians
$H_1$ and $H_2$ admit unique saddle points on
$\Pi_1\times\Pi_2$, denoted respectively by
\[
\Theta^+
\big(
\omega,t,P_1,P_2,\Lambda_1,\Gamma_1,\Gamma_2
\big)
=
(\Theta_1^+,\Theta_2^+)
\in\Pi_1\times\Pi_2
\]
and
\[
\Theta^-
\big(
\omega,t,P_1,P_2,\Lambda_2,\Gamma_1,\Gamma_2
\big)
=
(\Theta_1^-,\Theta_2^-)
\in\Pi_1\times\Pi_2.
\]
These saddle-point selectors admit predictable versions. Moreover, for some deterministic constant $C>0$,
\begin{equation}\label{Hamiltonian-two-sided-bound}
\left|
H_k^*
(\omega,t,P_1,P_2,\Lambda_k,\Gamma_1,\Gamma_2)
\right|
\le
C\bigl(1+|\Lambda_k|^2\bigr),
\qquad k=1,2,
\end{equation}
and
\begin{equation}\label{1203}
|\Theta^+|
\le
C(1+|\Lambda_1|),
\qquad
|\Theta^-|
\le
C(1+|\Lambda_2|).
\end{equation}
\end{proposition}
\begin{proof}
For notational simplicity, set
$
\mathcal H_k(v_1,v_2)
:=
H_k
(\omega,t,v_1,v_2,P_1,P_2,\Lambda_k,\Gamma_1,\Gamma_2).$ The bounds \eqref{2301}--\eqref{230181} imply that
$\mathcal H_k(\cdot,v_2)$ is uniformly convex and coercive
on $\Pi_1$ for every $v_2\in\Pi_2$, whereas
$\mathcal H_k(v_1,\cdot)$ is uniformly concave and
anti-coercive on $\Pi_2$ for every $v_1\in\Pi_1$.

To justify the preceding assertion, set
$
q_i(z):=P_i(t-)+\Gamma_i(t,z),
i=1,2,
$
and define
$
\phi_z(r)
:=
q_1(z)(r^+)^2+q_2(z)(r^-)^2.
$
For
$
\underline q(z):=q_1(z)\wedge q_2(z),
\overline q(z):=q_1(z)\vee q_2(z),
$
the functions
$
r\longmapsto
\phi_z(r)-\underline q(z)r^2
$
and
$
r\longmapsto
\overline q(z)r^2-\phi_z(r)
$
are convex. Hence, for $a,b\in\mathbb R$ and
$\lambda\in[0,1]$,
\[
\begin{aligned}
\underline q(z)\lambda(1-\lambda)|a-b|^2
\le
\lambda\phi_z(a)
+(1-\lambda)\phi_z(b)
-\phi_z(\lambda a+(1-\lambda)b)
\le
\overline q(z)\lambda(1-\lambda)|a-b|^2.
\end{aligned}
\]

Fix $v_2\in\Pi_2$. For $x_1,y_1\in\Pi_1$, the affine
terms in $v_1$ cancel, and the preceding inequality gives
\[
\begin{aligned}
&
\lambda\mathcal H_k(x_1,v_2)
+
(1-\lambda)\mathcal H_k(y_1,v_2)
-
\mathcal H_k
\big(
\lambda x_1+(1-\lambda)y_1,v_2
\big)
\\
&\ge
\lambda(1-\lambda)
(x_1-y_1)^\top
\Bigg[
M_{1k}
+
\int_{U_0}
\underline q(z)F_1(z)F_1(z)^\top\nu(dz)
\Bigg]
(x_1-y_1)
\\
&\ge
\eta\lambda(1-\lambda)|x_1-y_1|^2.
\end{aligned}
\]
Thus, $\mathcal H_k(\cdot,v_2)$ is uniformly strongly
convex on $\Pi_1$. Similarly, for fixed $v_1\in\Pi_1$, $\mathcal H_k(v_1,\cdot)$ is uniformly strongly
concave on $\Pi_2$.

Finally, the first-order inequalities for strongly convex and strongly concave functions yield
\[
\mathcal H_k(v_1,v_2)
\ge
\eta|v_1|^2-C_{v_2}|v_1|-C_{v_2} \ \text{and} \ 
\mathcal H_k(v_1,v_2)
\le
-\eta|v_2|^2+C_{v_1}|v_2|+C_{v_1}.
\]
Here $C_{v_2}>0$ (respectively, $C_{v_1}>0$) may depend on the fixed $v_2$ (respectively, $v_1$) and the frozen parameters, but is independent of the variable $v_1$ (respectively, $v_2$). Consequently,
$\mathcal H_k(\cdot,v_2)$ is coercive on $\Pi_1$, whereas $\mathcal H_k(v_1,\cdot)$ is anti-coercive on $\Pi_2$.

Since $\Pi_1$ and $\Pi_2$ are nonempty closed convex
subsets of finite-dimensional Euclidean spaces, the
coercivity and anti-coercivity established above, together
with
\cite[Chapter VI, Proposition 2.2]
{ekeland1999convex},
imply that $\mathcal H_k$ admits a saddle point and that
\[
\min_{v_1\in\Pi_1}
\max_{v_2\in\Pi_2}
\mathcal H_k(v_1,v_2)
=
\max_{v_2\in\Pi_2}
\min_{v_1\in\Pi_1}
\mathcal H_k(v_1,v_2).
\]
The uniform convexity in $v_1$ and uniform concavity in
$v_2$ imply that this saddle point is unique; see also
\cite[Chapter VI, Proposition 1.5]
{ekeland1999convex}.

Set
\[
\ell_{1k}:=\nabla_{v_1}\mathcal H_k(0,0),
\qquad
\ell_{2k}:=\nabla_{v_2}\mathcal H_k(0,0).
\]
The boundedness of the coefficients, the finiteness of $\nu(U_0)$, and
\eqref{postjump-bounds-section4} imply
\[
|\mathcal H_k(0,0)|\le C,
\qquad
|\ell_{1k}|+|\ell_{2k}|
\le C(1+|\Lambda_k|).
\]
Since $0\in\Pi_2$, uniform convexity in $v_1$ gives
\begin{align}
H_k^*
&\ge
\min_{v_1\in\Pi_1}\mathcal H_k(v_1,0)
\ge
\inf_{v_1\in\mathbb R^{m_1}}
\left\{
\mathcal H_k(0,0)+\langle\ell_{1k},v_1\rangle
+\eta|v_1|^2
\right\}
\notag\\
&=
\mathcal H_k(0,0)-\frac{|\ell_{1k}|^2}{4\eta}
\ge-C(1+|\Lambda_k|^2).
\label{071302}
\end{align}
Similarly, since $0\in\Pi_1$, uniform concavity in $v_2$ yields
\begin{align}
H_k^*
&\le
\max_{v_2\in\Pi_2}\mathcal H_k(0,v_2)
\le
\sup_{v_2\in\mathbb R^{m_2}}
\left\{
\mathcal H_k(0,0)+\langle\ell_{2k},v_2\rangle
-\eta|v_2|^2
\right\}
\notag\\
&=
\mathcal H_k(0,0)+\frac{|\ell_{2k}|^2}{4\eta}
\le C(1+|\Lambda_k|^2).
\label{071303}
\end{align}
This proves \eqref{Hamiltonian-two-sided-bound}.

For fixed $v_2\in\Pi_2$, let $\widehat v_{1,k}(v_2)$ be the unique
minimizer of $v_1\mapsto\mathcal H_k(v_1,v_2)$ over $\Pi_1$. Since
$0\in\Pi_1$, uniform convexity gives
\begin{equation}\label{202603202}
|\widehat v_{1,k}(v_2)|
\le
C\big(1+|v_2|+|\Lambda_k|\big).
\end{equation}
Indeed,
\[
\eta|\widehat v_{1,k}(v_2)|^2
\le
|\nabla_{v_1}\mathcal H_k(0,v_2)|
|\widehat v_{1,k}(v_2)|,
\]
and the gradient on the right-hand side is bounded by
$C(1+|v_2|+|\Lambda_k|)$.

Let
\[
g_k(v_2)
:=
\min_{v_1\in\Pi_1}
\mathcal H_k(v_1,v_2).
\]
Since ${\bf 0}\in\Pi_1$, the uniform concavity in $v_2$ and
the preceding estimates yield
\[
\begin{aligned}
g_k(v_2)
\le
\mathcal H_k({\bf 0},v_2)
\le
\mathcal H_k({\bf 0},{\bf 0})
+
\langle\ell_{2k},v_2\rangle
-
\eta|v_2|^2
\le
-\frac{\eta}{2}|v_2|^2
+
C\bigl(1+|\Lambda_k|^2\bigr).
\end{aligned}
\]
If $\widehat v_{2,k}$ is a maximizer of $g_k$, then
$g_k(\widehat v_{2,k})=H_k^*$. Combining the preceding
estimate with \eqref{071302}, we obtain
$
\frac{\eta}{2}
|\widehat v_{2,k}|^2
\le
C\bigl(1+|\Lambda_k|^2\bigr).
$
Consequently, for some deterministic constant $R>0$,
\begin{equation}\label{202303203}
\operatorname*{arg\,max}_{v_2\in\Pi_2}g_k(v_2)
\subseteq
\left\{
v_2\in\Pi_2:
|v_2|\le R\bigl(1+|\Lambda_k|\bigr)
\right\}.
\end{equation}
For $k=1$, this gives $|\Theta_2^+|\le C(1+|\Lambda_1|)$, and
\eqref{202603202} then yields the same estimate for $\Theta_1^+$.
The argument for $k=2$ is identical and proves \eqref{1203}.

The predictability of $\Theta^+$ and $\Theta^-$ follows from \eqref{202603202}--\eqref{202303203} and the measurable maximum theorem for predictable Carathéodory integrands.
\end{proof}

\begin{lemma}\label{Lemma4.1}
Under the assumptions of Proposition \ref{prop:hamiltonian-saddle},
$\int_0^{\cdot}\Lambda_k(t)dW(t)$ is a BMO martingale for $k=1,2$.
\end{lemma}

\begin{proof}
Let $f_k$ denote the generator of the $k$th equation in \eqref{csre1}.
By the boundedness of $P_k$, the coefficients, and
\eqref{Hamiltonian-two-sided-bound},
\[
|f_k(t)|\le C\big(1+|\Lambda_k(t)|^2\big).
\]
Fix a stopping time $\tau\le T$ and let $\alpha>0$.
Applying It\^o's formula with jumps to
$e^{\alpha P_k(t)}$ on $[\tau,T]$, we obtain
\begin{align}
e^{\alpha P_k(T)}
-
e^{\alpha P_k(\tau)}
&=
\int_\tau^T
e^{\alpha P_k(t-)}
\left[
-\alpha f_k(t)
+
\frac{\alpha^2}{2}
|\Lambda_k(t)|^2
\right]dt
\notag\\
&\quad+
\int_\tau^T
\int_{U_0}
e^{\alpha P_k(t-)}
\left[
e^{\alpha\Gamma_k(t,z)}
-
1
-
\alpha\Gamma_k(t,z)
\right]
\nu(dz)dt
\notag\\
&\quad+
M_k^W(T)-M_k^W(\tau)
+
M_k^N(T)-M_k^N(\tau),
\label{eq:BMO-exponential-Ito}
\end{align}
where
\[
M_k^W(t)
:=
\int_0^t
\alpha e^{\alpha P_k(s-)}
\Lambda_k(s)dW(s)
\]
and
\[
M_k^N(t)
:=
\int_0^t\int_{U_0}
e^{\alpha P_k(s-)}
\left(
e^{\alpha\Gamma_k(s,z)}-1
\right)
\widetilde N(ds,dz).
\]

We next verify that $M_k^W$ and $M_k^N$ are
square-integrable martingales. Since
$
0\le P_k(t-)\le K,
$
we have
\[
\mathbb E\int_0^T
\left|
\alpha e^{\alpha P_k(t-)}
\Lambda_k(t)
\right|^2dt
\le
\alpha^2e^{2\alpha K}
\mathbb E\int_0^T
|\Lambda_k(t)|^2dt
<\infty.
\]
Hence, $M_k^W$ is a square-integrable martingale.

Moreover, by \eqref{postjump-bounds-section4}, we have 
$
0<
P_k(t-)+\Gamma_k(t,z)
\le K.
$
Consequently,
\[
\begin{aligned}
&
\left|
e^{\alpha P_k(t-)}
\left(
e^{\alpha\Gamma_k(t,z)}-1
\right)
\right|
=
\left|
e^{\alpha(P_k(t-)+\Gamma_k(t,z))}
-
e^{\alpha P_k(t-)}
\right|
\le
2e^{\alpha K}.
\end{aligned}
\]
Since $\nu(U_0)<\infty$,
\[
\begin{aligned}
&\mathbb E\int_0^T\int_{U_0}
\left|
e^{\alpha P_k(t-)}
\left(
e^{\alpha\Gamma_k(t,z)}-1
\right)
\right|^2
\nu(dz)dt
\le
4e^{2\alpha K}T\nu(U_0)
<\infty.
\end{aligned}
\]
Thus, $M_k^N$ is also a square-integrable martingale.
In particular,
\[
\mathbb E\left[
\left.
M_k^W(T)-M_k^W(\tau)
\right|
\mathcal F_\tau
\right]
=
0 \ \text{and}\
\mathbb E\left[
\left.
M_k^N(T)-M_k^N(\tau)
\right|
\mathcal F_\tau
\right]
=
0.
\]

Since
$
e^{\alpha x}-1-\alpha x\ge0, x\in\mathbb R,
$
the jump-compensator term in
\eqref{eq:BMO-exponential-Ito} is nonnegative. 
Therefore,
\begin{align}
e^{\alpha P_k(T)}
-
e^{\alpha P_k(\tau)}
&\ge
\int_\tau^T
e^{\alpha P_k(t-)}
\left[
\left(
\frac{\alpha^2}{2}
-
C\alpha
\right)
|\Lambda_k(t)|^2
-
C\alpha
\right]dt
\notag\\
&\quad+
M_k^W(T)-M_k^W(\tau)
+
M_k^N(T)-M_k^N(\tau).
\label{eq:BMO-lower-estimate}
\end{align}

Choose $\alpha>2C$. Taking the conditional expectation
with respect to $\mathcal F_\tau$ in
\eqref{eq:BMO-lower-estimate}, we obtain
\[
\begin{aligned}
\left(
\frac{\alpha^2}{2}
-
C\alpha
\right)
\mathbb E\left[
\left.
\int_\tau^T
e^{\alpha P_k(t-)}
|\Lambda_k(t)|^2dt
\right|
\mathcal F_\tau
\right]
&\le
\mathbb E\left[
\left.
e^{\alpha P_k(T)}
-
e^{\alpha P_k(\tau)}
\right|
\mathcal F_\tau
\right]
+
C\alpha
\mathbb E\left[
\left.
\int_\tau^T
e^{\alpha P_k(t-)}dt
\right|
\mathcal F_\tau
\right]
\\
&\qquad\le
e^{\alpha K}
+
C\alpha T e^{\alpha K}.
\end{aligned}
\]
Since $P_k(t-)\ge0$, we have
$
e^{\alpha P_k(t-)}\ge1.
$
It follows that
\begin{equation}\label{BMO-conditional-estimate}
\mathbb E\left[
\left.
\int_\tau^T
|\Lambda_k(t)|^2dt
\right|
\mathcal F_\tau
\right]
\le
C_\Lambda,
\qquad k=1,2,
\end{equation}
where
\[
C_\Lambda
:=
\frac{
e^{\alpha K}
\left(
1+C\alpha T
\right)
}{
\frac{\alpha^2}{2}-C\alpha
}
\]
is independent of $\tau$. Hence,
$\int_0^\cdot\Lambda_k(t)dW(t)$ is a BMO martingale.
\end{proof}

\begin{theorem}\label{Main}
Let Assumptions \ref{A1} and \ref{A33} hold. Suppose that
\eqref{csre1} admits a solution satisfying
\[
\mathbb P\big(
0<P_k(t)\le K\text{ for all }t\in[0,T]
\big)=1,
\qquad k=1,2.
\]
Then Problem (C-ZLQJ) has the feedback representation
\begin{equation*}\label{main11130}
u^*(t)
=
\Theta^+(t)(X^*(t-))^+
+
\Theta^-(t)(X^*(t-))^-.
\end{equation*}
Moreover,
\begin{eqnarray}\label{071206}
    J(x;u_1^*,u_2^*)
=
\mathbb E\left[
P_1(0)(x^+)^2+P_2(0)(x^-)^2
\right].
\end{eqnarray}
\end{theorem}

Before proving Theorem \ref{Main}, we establish two auxiliary results.

\begin{lemma}\label{lem2}
Under the assumptions of Theorem \ref{Main}, then the closed-loop state equation
\[
\left\{
\begin{aligned}
d\bar X(t)
={}&
\big(
A\bar X
+
B^\top\bar u
\big)dt
+
\big(
C\bar X
+
D^\top\bar u
\big)dW(t)
+
\int_{U_0}
\big(
E(z)\bar X
+
F(z)^\top\bar u
\big)
\widetilde N(dt,dz),
\\
\bar X(0)
={}&x,
\end{aligned}
\right.
\]
admits a unique c\`adl\`ag adapted solution with control pair $\bar u(t)
=
\Theta^+(t)(\bar X(t-))^+
+
\Theta^-(t)(\bar X(t-))^-.
$
\end{lemma}

\begin{proof}
By Proposition \ref{prop:hamiltonian-saddle}, the selectors are
predictable and satisfy \eqref{1203}. Since
$\nu(U_0)<\infty$, the Poisson random measure has finitely many atoms on
$[0,T]\times U_0$, $\mathbb P$-a.s. Outside a fixed null set, write them
as
\[
(T_1,Z_1),\ldots,(T_{N_T},Z_{N_T}),
\qquad
0<T_1<\cdots<T_{N_T}\le T,
\]
and set $T_0:=0$, $T_{N_T+1}:=T$, and $\xi_0:=x$.

Suppose that an $\mathcal F_{T_j}$-measurable random variable $\xi_j$ has
been defined. On $[T_j,T_{j+1})$, set
\[
\begin{aligned}
a_+
&:=A+B^\top\Theta^+
-
\int_{U_0}
\big(E(z)+F(z)^\top\Theta^+\big)\nu(dz),\ 
c_+
:=C+D^\top\Theta^+,
\\
a_-
&:=A-B^\top\Theta^-
-
\int_{U_0}
\big(E(z)-F(z)^\top\Theta^-\big)\nu(dz),\ 
c_-
:=C-D^\top\Theta^-.
\end{aligned}
\]
The selector estimates and $\Lambda_1,\Lambda_2\in
L_{\mathbb F}^2(0,T;\mathbb R)$ imply
\[
\int_{T_j}^{T_{j+1}}
\big(
|a_+(t)|+|a_-(t)|+|c_+(t)|^2+|c_-(t)|^2
\big)dt<\infty,
\qquad \mathbb P\text{-a.s.}
\]
Define, for $T_j\le t<T_{j+1}$,
\[
\begin{aligned}
\bar X_{j,+}(t)
&:=(\xi_j)^+
\exp\left\{
\int_{T_j}^t
\left(a_+(s)-\frac12|c_+(s)|^2\right)ds
+
\int_{T_j}^t c_+(s)dW(s)
\right\},
\\
\bar X_{j,-}(t)
&:=(\xi_j)^-
\exp\left\{
\int_{T_j}^t
\left(a_-(s)-\frac12|c_-(s)|^2\right)ds
+
\int_{T_j}^t c_-(s)dW(s)
\right\}.
\end{aligned}
\]
The exponential factors are strictly positive. Since
$(\xi_j)^+(\xi_j)^-=0$, one of these two processes is identically zero.
Thus, with
\[
\bar X(t):=\bar X_{j,+}(t)-\bar X_{j,-}(t),
\]
we have
$(\bar X(t))^+=\bar X_{j,+}(t)$ and
$(\bar X(t))^-=\bar X_{j,-}(t)$. Subtracting the two scalar linear
equations gives the continuous part of the closed-loop dynamics on
$[T_j,T_{j+1})$.

If $j+1\le N_T$, define
\[
\begin{aligned}
\xi_{j+1}
:={}&\bar X(T_{j+1}-)
+E(T_{j+1},Z_{j+1})\bar X(T_{j+1}-)
\\
&+F_1(T_{j+1},Z_{j+1})^\top\bar u_1(T_{j+1})
+F_2(T_{j+1},Z_{j+1})^\top\bar u_2(T_{j+1}),
\end{aligned}
\]
and set $\bar X(T_{j+1})=\xi_{j+1}$. Iterating over the finitely many
jump intervals yields a global c\`adl\`ag adapted solution.

On each interval $[T_j,T_{j+1})$, the closed-loop drift and diffusion
coefficients are Lipschitz in the state with an a.s. integrable random
Lipschitz modulus. Hence pathwise uniqueness holds on that interval. The
jump update is uniquely determined by the pre-jump state and the Poisson
mark. Induction over the jump times proves global uniqueness.
\end{proof}

\begin{lemma}\label{lem:feedback-admissibility}
Let the assumptions of Theorem \ref{Main} hold, and let $(\bar X,\bar u)$ be the closed-loop state--control pair constructed in Lemma \ref{lem2}. Then $
\bar u\in\mathcal U_1\times\mathcal U_2
\ \text{and}\ 
\bar X\in\mathcal S_{\mathbb F}^2(0,T;\mathbb R).$
\end{lemma}

\begin{proof}
Since $\Theta_i^+,\Theta_i^-\in\Pi_i$,
$(\bar X(t-))^\pm\ge0$, and $\Pi_i$ is a convex cone,
$\bar u_i(t)\in\Pi_i$. Predictability follows from the predictability of
$\bar X(t-)$ and the selectors.

For $n\ge1$, set
\begin{equation*}\label{feedback-localizing-time}
\tau_n
:=
\inf\{t\in[0,T]:|\bar X(t)|\ge n\}\wedge T.
\end{equation*}
Since $\bar X$ is c\`adl\`ag, $\tau_n\uparrow T$ a.s. By
\eqref{1203},
\begin{equation*}\label{feedback-growth-estimate}
|\bar u_1(t)|^2+|\bar u_2(t)|^2
\le
C\big(1+|\Lambda_1(t)|^2+|\Lambda_2(t)|^2\big)
|\bar X(t-)|^2.
\end{equation*}
Consequently,
\begin{equation}\label{localized-feedback-integrability}
\mathbb E\int_0^{\tau_n}
\big(|\bar u_1(t)|^2+|\bar u_2(t)|^2\big)dt<\infty,
\end{equation}
and the standard stopped-state estimate gives
\begin{equation}\label{localized-state-S2}
\mathbb E\sup_{0\le t\le T}|\bar X(t\wedge\tau_n)|^2<\infty.
\end{equation}
In addition,
\begin{equation}\label{localized-Lambda-state-integrability}
\mathbb E\int_0^{\tau_n}
\big(|\Lambda_1(t)|^2+|\Lambda_2(t)|^2\big)
|\bar X(t-)|^2dt<\infty.
\end{equation}
For the remainder of the proof, set 
\begin{eqnarray}\label{eq072102}
 \bar Y(t)
:=
P_1(t)(\bar X(t))^+
-
P_2(t)(\bar X(t))^-.   
\end{eqnarray}
 By the product formula  and Meyer--It\^o's formula with jumps, we obtain  
\begin{eqnarray*}
  d\bar Y(t)
=
\bar B_Y(t)dt+\bar Z(t)dW(t)
+\int_{U_0}\bar K(t,z)\widetilde N(dt,dz)
+\frac12\big(P_1(t-)-P_2(t-)\big)dL_t^0(\bar X),  
\end{eqnarray*}
where
\begin{eqnarray*}
  \bar B_Y(t)
&=&
\big(I_{\bar{X}>0}(t)P_1+I_{\bar{X}\leq0}P_2\big)
\Big(A\bar X+B^\top\bar u
\Big)-\bar X^+
\Big(
(2A+C^2)P_1+2C\Lambda_1+Q+H_1^*
\Big)\\
&&+\bar X^-
\big(
(2A+C^2)P_2+2C\Lambda_2+Q+H_2^*
\big)+\big(I_{\bar{X}>0}\Lambda_1+I_{\bar{X}\leq0}\Lambda_2\big)
\big(C\bar X+D^\top\bar u\big)\\
&&+\int_{U_0}
\Big(
\bar K
-\bar X^+\Gamma_1
+\bar X^-\Gamma_2
-\big(I_{\bar{X}>0}P_1+I_{\bar{X}\leq0}P_2\big)
\big(E\bar X+F^\top\bar u
\big)
\Big)\nu(dz), 
\end{eqnarray*}
\begin{eqnarray}\label{eq072104}
 \bar Z&=&
I_{\bar{X}>0}P_1
\big(C\bar X+D^\top\bar u
\big)+
I_{\bar{X}\leq0}P_2
\big(C\bar X+D^\top\bar u
\big)+\Lambda_1\bar X^+
-\Lambda_2\bar X^-.   
\end{eqnarray}
and 
\begin{eqnarray}\label{eq072105}
  \bar K
&=&
(P_1+\Gamma_1)
\big(\bar X+E\bar X+F^\top\bar u\big)^+-
(P_2+\Gamma_2)
\big(\bar X+E\bar X+F^\top\bar u
\big)^--
P_1\bar X^+
+
P_2\bar X^-. \  
\end{eqnarray}
By
\eqref{localized-feedback-integrability},
\eqref{localized-state-S2},
\eqref{localized-Lambda-state-integrability}, and
\eqref{postjump-bounds-section4}, we get 
$
\bar Z\mathbf 1_{[0,\tau_n]}
\in
L_{\mathbb F}^2(0,T;\mathbb R) \ \text{and} \ 
\bar{\mathcal K}\mathbf 1_{[0,\tau_n]}
\in
\mathcal L_{\mathcal P}^{2,\nu}
(0,T;\mathbb R).
$
Hence, On $[0,\tau_n]$, the IESREJs together with the saddle-point equalities
$
\cH_1(\Theta_1^+,\Theta_2^+)=H_1^*,\ 
\cH_2(\Theta_1^-,\Theta_2^-)=H_2^*,
$
give
\begin{align*}
d\bar Y(t)
={}&-
\left(
A\bar Y+C\bar Z
+\int_{U_0}E(z)\bar{\mathcal K}(z)\nu(dz)
+Q\bar X
\right)dt+\bar Z(t)dW(t)
\notag\\
&
+\int_{U_0}\bar{\mathcal K}(z)\widetilde N(dt,dz)
+\frac12(P_1(t-)-P_2(t-))dL_t^0(\bar X).
\label{localized-adjoint-with-local-time}
\end{align*}
Here $L^0(\bar X)$ denotes the local time generated by the
continuous semimartingale part of $\bar X$. On
$[0,\tau_n]$,
\[
|C(t)\bar X(t)+D^\top(t)\bar u(t)|^2
\le
C\big(
1+|\Lambda_1(t)|^2+|\Lambda_2(t)|^2
\big)
|\bar X(t-)|^2.
\]
Hence, by the occupation-density formula,
\[
\begin{aligned}
\mathbb E L_{t\wedge\tau_n}^0(\bar X)
&=
\lim_{\varepsilon\downarrow0}
\frac{1}{2\varepsilon}
\mathbb E\int_0^{t\wedge\tau_n}
\mathbf 1_{\{|\bar X(s-)|\le\varepsilon\}}
|C(s)\bar X(s)+D^\top(s)\bar u(s)|^2ds
\\
&\le
C\lim_{\varepsilon\downarrow0}
\varepsilon
\mathbb E\int_0^T
\big(
1+|\Lambda_1(s)|^2+|\Lambda_2(s)|^2
\big)ds
=0.
\end{aligned}
\]
Therefore,
$
L_{t\wedge\tau_n}^0(\bar X)=0,\
0\le t\le T,\ \mathbb P\text{-a.s.},
$  and 
\begin{equation}\label{localized-feedback-adjoint}
\left\{
\begin{aligned}
d\bar Y(t)
={}&-
\left(
A\bar Y+C\bar Z
+\int_{U_0}E(z)\bar{\mathcal K}(z)\nu(dz)
+Q\bar X
\right)dt
+\bar Z(t)dW(t)
+\int_{U_0}\bar{\mathcal K}(z)\widetilde N(dt,dz),
\\
\bar Y(\tau_n)
={}&P_1(\tau_n)(\bar X(\tau_n))^+
-P_2(\tau_n)(\bar X(\tau_n))^-.
\end{aligned}
\right.
\end{equation}

Define
\begin{align*}
\bar{\mathcal G}_1
:={}&
B_1\bar Y+D_1\bar Z
+R_{11}\bar u_1+R_{12}\bar u_2
+\int_{U_0}F_1(z)\bar{\mathcal K}(z)\nu(dz),
\\
\bar{\mathcal G}_2
:={}&
B_2\bar Y+D_2\bar Z
+R_{21}\bar u_1+R_{22}\bar u_2
+\int_{U_0}F_2(z)\bar{\mathcal K}(z)\nu(dz).
\end{align*}
On $\{\bar X(t-)>0\}$, for $v_i\in\Pi_i$,
\begin{align*}
\langle\bar{\mathcal G}_1,v_1-\bar u_1\rangle
&=
\frac{|\bar X(t-)|^2}{2}
\left\langle
\partial_{v_1}\cH_1(\Theta_1^+,\Theta_2^+),
\frac{v_1}{\bar X(t-)}-\Theta_1^+
\right\rangle\ge0,
\\
\langle\bar{\mathcal G}_2,v_2-\bar u_2\rangle
&=
\frac{|\bar X(t-)|^2}{2}
\left\langle
\partial_{v_2}\cH_1(\Theta_1^+,\Theta_2^+),
\frac{v_2}{\bar X(t-)}-\Theta_2^+
\right\rangle\le0.
\end{align*}
The same argument with $\cH_2$ and $\Theta^-$ applies on
$\{\bar X(t-)<0\}$. On $\{\bar X(t-)=0\}$, all the displayed
quantities vanish. Thus, on $[0,\tau_n]$,
\begin{equation}\label{localized-feedback-VI}
\langle\bar{\mathcal G}_1,v_1-\bar u_1\rangle\ge0,
\qquad
\langle\bar{\mathcal G}_2,v_2-\bar u_2\rangle\le0,
\qquad
(v_1,v_2)\in\Pi_1\times\Pi_2.
\end{equation}

For $\xi\in\mathbb R$, define
$
\Phi_n(\xi)
:=P_1(\tau_n)(\xi^+)^2+P_2(\tau_n)(\xi^-)^2.
$
Then $\Phi_n$ is convex and continuously differentiable, and
\begin{equation}\label{localized-terminal-estimate}
0
\le
\Phi_n(\xi+h)-\Phi_n(\xi)-\Phi_n'(\xi)h
\le
K|h|^2,
\qquad \xi,h\in\mathbb R.
\end{equation}
For predictable controls $v=(v_1,v_2)$ on $[0,\tau_n]$, define
\[
\mathcal J_n(x;v_1,v_2)
:=
\mathbb E\left[
\int_0^{\tau_n}
\big(Q|X_x^{v_1,v_2}(t)|^2+v(t)^\top R(t)v(t)\big)dt
+
\Phi_n(X_x^{v_1,v_2}(\tau_n))
\right].
\]
The generalized It\^o formula applied to
$P_1(t)((\bar{X}(t))^+)^2$ and $P_2(t)((\bar{X}(t))^-)^2$ gives
\begin{equation*}\label{061661}
\begin{aligned}
\mathcal J_n(x;\bar u_1,\bar u_2)
={}&
\mathbb E\left[P_1(0)(x^+)^2+P_2(0)(x^-)^2\right]
\\
&+
\mathbb E\int_0^{\tau_n}
\psi
(\omega,t,\bar X,\bar u_1,\bar u_2,
P_1,P_2,\Lambda_1,\Lambda_2,\Gamma_1,\Gamma_2)dt,
\end{aligned}
\end{equation*}
where, with
$I_+(t):=\mathbf 1_{\{X(t-)>0\}}$ and
$I_-(t):=\mathbf 1_{\{X(t-)\le0\}}$,
\begin{align}
&\psi
(\omega,t,X,u_1,u_2,P_1,P_2,\Lambda_1,\Lambda_2,\Gamma_1,\Gamma_2)
\notag\\
&=
 u_1^\top
\Big(R_{11}+(I_+P_1+I_-P_2)D_1D_1^\top\Big)u_1+
 u_2^\top
\Big(R_{22}+(I_+P_1+I_-P_2)D_2D_2^\top\Big)u_2
\notag\\
&\quad+
2u_1^\top
\Big((I_+P_1+I_-P_2)D_1D_2^\top+R_{12}\Big)u_2
\notag\\
&\quad+
2I_+u_1^\top(P_1B_1+P_1D_1C+D_1\Lambda_1)X^+
-2I_-u_1^\top(P_2B_1+P_2D_1C+D_1\Lambda_2)X^-
\notag\\
&\quad+
2I_+u_2^\top(P_1B_2+P_1D_2C+D_2\Lambda_1)X^+
-2I_-u_2^\top(P_2B_2+P_2D_2C+D_2\Lambda_2)X^-
\notag\\
&\quad+
\big(2P_2I_-X^- -2P_1I_+X^+\big)
\int_{U_0}
\big(E(z)X+F_1(z)^\top u_1+F_2(z)^\top u_2\big)\nu(dz)
\notag\\
&\quad+
\int_{U_0}
(P_1+\Gamma_1(z))
\left[
\left(
\big(
X+E(z)X+F_1(z)^\top u_1+F_2(z)^\top u_2
\big)^+
\right)^2
-(X^+)^2
\right]\nu(dz)
\notag\\
&\quad+
\int_{U_0}
(P_2+\Gamma_2(z))
\left[
\left(
\big(
X+E(z)X+F_1(z)^\top u_1+F_2(z)^\top u_2
\big)^-
\right)^2
-(X^-)^2
\right]\nu(dz)
\notag\\
&\quad-
H_1^*(\omega,t,P_1,P_2,\Lambda_1,\Gamma_1,\Gamma_2)(X^+)^2
-
H_2^*(\omega,t,P_1,P_2,\Lambda_2,\Gamma_1,\Gamma_2)(X^-)^2.
\label{02031}
\end{align}
Observe  (\ref{02031}), we define 
 $$v(t):=(v_1(t)^\top,v_2(t)^\top)^\top=
\left\{
\begin{array}{ll}
(\frac{u_1^\top(t)}{|X(t-)|}, \frac{u_2^\top(t)}{|X(t-)|})^\top, \  \text{if}\  |X(t-)|>0, \\
(0,0)^\top, \text{if}\  |X(t-)|=0, \\
\end{array}
\right.$$
 It is clear that $v$ is valued in $\Pi_1\times \Pi_2$. Hence, if $X(t-)>0$, then 
 \begin{eqnarray*}\label{11301}
 \psi(\omega, t,X, u_1, u_2,P_1, P_2, \Lambda_1,\Lambda_2,\Gamma_1, \Gamma_2)
   =\bigg\{H_1(
\omega,t,v_1,v_2,
P_1,P_2,\Lambda_1,\Gamma_1,\Gamma_2)-H_1^*(\omega,t, P_1, P_2,\Lambda_1, \Gamma_1,\Gamma_2)\bigg\}X^2.
 \end{eqnarray*}
If $X(t-)\leq 0,$ then
 \begin{eqnarray*}
\psi(\omega, t,X, u_1, u_2,P_1, P_2, \Lambda_1,\Lambda_2,\Gamma_1, \Gamma_2)
=\bigg\{H_2(
\omega,t,v_1,v_2,
P_1,P_2,\Lambda_2,\Gamma_1,\Gamma_2)-H_2^*(\omega,t, P_1, P_2,\Lambda_2, \Gamma_1,\Gamma_2)\bigg\}X^2.
  \end{eqnarray*}
if $X(t-)=0$, then 
\begin{eqnarray*}
&&\psi(\omega, t, u_1, u_2,P_1, P_2, \Lambda_1,\Lambda_2,\Gamma_1, \Gamma_2)\nonumber\\
&=&u_1^\top \big(R_{11}+P_2D_1 D_1^\top\big)u_1 +u_2^\top \big(R_{22}+P_2D_2 D_2^\top \big)u_2+2u_1^\top (P_2D_1 D_2^\top+R_{12}) u_2\nonumber\\
&&+\int_{U_0}(P_1+\Gamma_1(z))\big((F_1^\top (z)u_1+F_2^\top(z)u_2\big)^+)^2\nu(dz)\nonumber\\
 && +\int_{U_0}(P_2+\Gamma_2(z))\big((F_1^\top(z)u_1+F_2^\top(z)u_2)^-\big)^2\nu(dz).
\end{eqnarray*}
Hence, for the feedback pair $(\bar{X}, \bar{u})$,
\[
\begin{aligned}
\psi(\omega, t,\bar{X}, \bar{u}_1, \bar{u}_2,P_1, P_2, \Lambda_1,\Lambda_2,\Gamma_1, \Gamma_2)
&=(\bar{X}^+)^2\big(\cH_1(\Theta_1^+,\Theta_2^+)-H_1^*\big)\notag\\
&\quad+(\bar{X}^-)^2\big(\cH_2(\Theta_1^-,\Theta_2^-)-H_2^*\big)
=0.
\end{aligned}
\]
Therefore,
\begin{equation}\label{localized-feedback-value}
\mathcal J_n(x;\bar u_1,\bar u_2)
=
\mathbb E\left[P_1(0)(x^+)^2+P_2(0)(x^-)^2\right].
\end{equation}

Let $u_1$ be admissible on $[0,\tau_n]$ and set
$
\Delta X_1(t)
:=X_x^{u_1,\bar u_2}(t)-\bar X(t).
$
Applying It\^o's formula to $\bar Y\Delta X_1$, and using
\eqref{localized-feedback-adjoint}, gives
\begin{align*}
&\mathcal J_n(x;u_1,\bar u_2)
-
\mathcal J_n(x;\bar u_1,\bar u_2)\notag\\
&=
2\mathbb E\int_0^{\tau_n}
\langle\bar{\mathcal G}_1,u_1-\bar u_1\rangle dt+
\mathbb E\int_0^{\tau_n}
\left[
Q|\Delta X_1|^2
+(u_1-\bar u_1)^\top R_{11}(u_1-\bar u_1)
\right]dt
\notag\\
&\quad+
\mathbb E\Big[
\Phi_n(\bar X(\tau_n)+\Delta X_1(\tau_n))
-\Phi_n(\bar X(\tau_n))
-\Phi_n'(\bar X(\tau_n))\Delta X_1(\tau_n)
\Big].
\label{localized-player1-expansion}
\end{align*}
By \eqref{localized-feedback-VI},
\eqref{localized-terminal-estimate}, $Q\ge0$, and
$R_{11}\succeq\underline\delta I_{m_1}$,
\begin{equation*}\label{localized-player1-strong-convexity}
\mathcal J_n(x;u_1,\bar u_2)
-
\mathcal J_n(x;\bar u_1,\bar u_2)
\ge
\underline\delta
\mathbb E\int_0^{\tau_n}|u_1(t)-\bar u_1(t)|^2dt.
\end{equation*}
Taking $u_1=0$ yields
\begin{equation*}\label{localized-player1-zero-comparison}
\mathcal J_n(x;\bar u_1,\bar u_2)
\le
\mathcal J_n(x;0,\bar u_2)
-
\underline\delta
\mathbb E\int_0^{\tau_n}|\bar u_1(t)|^2dt.
\end{equation*}
Since $\Phi_n(\xi)\le K|\xi|^2$, the estimate
\eqref{player2-state-cost-estimate} and
$R_{22}\preceq-(\bar\delta+K\bar c)I_{m_2}$ imply
\begin{equation*}\label{localized-player2-energy}
\mathcal J_n(x;0,\bar u_2)
\le
K\bar c|x|^2
-
\bar\delta
\mathbb E\int_0^{\tau_n}|\bar u_2(t)|^2dt.
\end{equation*}
The right-hand side of \eqref{localized-feedback-value} is nonnegative.
Combining the last three estimates gives
\begin{equation*}\label{uniform-feedback-energy}
\underline\delta
\mathbb E\int_0^{\tau_n}|\bar u_1(t)|^2dt
+
\bar\delta
\mathbb E\int_0^{\tau_n}|\bar u_2(t)|^2dt
\le
K\bar c|x|^2.
\end{equation*}
Letting $n\to\infty$ and using monotone convergence, we obtain
\[
\mathbb E\int_0^T
\big(|\bar u_1(t)|^2+|\bar u_2(t)|^2\big)dt<\infty.
\]
Thus $\bar u\in\mathcal U_1\times\mathcal U_2$. The standard estimate for
the state equation then gives
\[
\mathbb E\sup_{0\le t\le T}|\bar X(t)|^2
\le
C\left[
|x|^2+
\mathbb E\int_0^T
\big(|\bar u_1(t)|^2+|\bar u_2(t)|^2\big)dt
\right]
<\infty,
\]
so $\bar X\in\mathcal S_{\mathbb F}^2(0,T;\mathbb R)$.
\end{proof}

\begin{proof}[Proof of Theorem \ref{Main}]
Let $(\bar X,\bar u)$ be the closed-loop state--control pair
constructed above. By Lemmas \ref{lem2} and
\ref{lem:feedback-admissibility},
 we have 
$\bar u\in\mathcal U_1\times\mathcal U_2,\ 
\bar X\in\mathcal S_{\mathbb F}^2(0,T;\mathbb R).
$
Recall the processes
$\bar Y,\bar Z,\bar{\mathcal K}$ and
$\bar{\mathcal G}_i$, $i=1,2$, defined above.

We first verify their global integrability. Set
\[
\mathcal A(t)
:=
\int_0^t
\big(
|\Lambda_1(s)|^2+|\Lambda_2(s)|^2
\big)ds
\]
and, for $\lambda\ge0$,
\[
\rho_\lambda
:=
\inf\left\{
t\in[0,T]:
|\bar X(t)|^2>\lambda
\right\}\wedge T.
\]
Since $\bar X$ is c\`adl\`ag, we have 
$
\mathbf 1_{\{|\bar X(t-)|^2>\lambda\}}
\le
\mathbf 1_{\{\rho_\lambda<t\}},\
d\mathbb P\otimes dt\text{-a.e.}
$
Hence, by the conditional BMO estimate
\eqref{BMO-conditional-estimate},
\[
\begin{aligned}
\mathbb E\int_0^T
\mathbf 1_{\{|\bar X(t-)|^2>\lambda\}}d\mathcal A(t)
&\le
\mathbb E\left[
\mathbf 1_{\{\rho_\lambda<T\}}
\mathbb E\left[
\left.
\mathcal A(T)-\mathcal A(\rho_\lambda)
\right|
\mathcal F_{\rho_\lambda}
\right]
\right]
\le
C_\Lambda
\mathbb P\left(
\sup_{0\le t\le T}|\bar X(t)|^2>\lambda
\right).
\end{aligned}
\]
The layer-cake formula and Tonelli's theorem therefore give
\begin{equation*}\label{global-Lambda-X-estimate}
\begin{aligned}
&\mathbb E\int_0^T
\big(
|\Lambda_1(t)|^2+|\Lambda_2(t)|^2
\big)|\bar X(t-)|^2dt
\le
C_\Lambda
\mathbb E\sup_{0\le t\le T}|\bar X(t)|^2
<\infty.
\end{aligned}
\end{equation*}
Recall that $\bar{Y}, \bar{X}$ and $\bar{\cK}$ were defined in
\eqref{eq072102}, \eqref{eq072104} and \eqref{eq072105}, the
boundedness of the coefficients, and
\eqref{postjump-bounds-section4}, gives
$
|\bar Y(t)|
\le
K|\bar X(t)|
$
and
\[
\begin{aligned}
|\bar Z(t)|^2
+
\int_{U_0}
|\bar{\mathcal K}(t,z)|^2\nu(dz)
\le
C\Big[
&\big(
1+|\Lambda_1(t)|^2+|\Lambda_2(t)|^2
\big)|\bar X(t-)|^2
+
|\bar u_1(t)|^2+|\bar u_2(t)|^2
\Big].
\end{aligned}
\]
Consequently,
$
(\bar Y,\bar Z,\bar{\mathcal K})
\in
\mathcal S_{\mathbb F}^2(0,T;\mathbb R)
\times
L_{\mathbb F}^2(0,T;\mathbb R)
\times
\mathcal L_{\mathcal P}^{2,\nu}(0,T;\mathbb R).
$

We next pass to the limit in the localized adjoint
relations. Since $\tau_n\uparrow T$ and every c\`adl\`ag
path on $[0,T]$ is bounded,then
$
\bar Y(\tau_n)\to\bar Y(T),\ \text{in }L^2(\Omega).
$
Moreover, the preceding integrability implies that the drift
terms in \eqref{localized-feedback-adjoint} converge in
$L^1$, whereas the Brownian and Poisson stochastic integrals
converge in $L^2$ by their respective isometries. Letting
$n\to\infty$ in \eqref{localized-feedback-adjoint}, we obtain
\begin{equation*}\label{global-feedback-adjoint}
\left\{
\begin{aligned}
d\bar Y(t)
={}&-
\Big(
A\bar Y
+C\bar Z
+Q\bar X
+
\int_{U_0}
E(z)\bar{\mathcal K}(z)\nu(dz)
\Big)dt
+
\bar ZdW(t)
+
\int_{U_0}
\bar{\mathcal K}(z)\widetilde N(dt,dz),
\\
\bar Y(T)
={}&G(T)\bar X(T).
\end{aligned}
\right.
\end{equation*}
Here the terminal condition follows from
$P_1(T)=P_2(T)=G(T)$.

Likewise, since
$\mathbf 1_{\{t\le\tau_n\}}\uparrow1$
$d\mathbb P\otimes dt$-a.e.,
\eqref{localized-feedback-VI} yields
\begin{equation*}\label{global-feedback-VI}
\left\{
\begin{aligned}
\left\langle
\bar{\mathcal G}_1(t),
v_1-\bar u_1(t)
\right\rangle
&\ge0,
&&v_1\in\Pi_1,
\\
\left\langle
\bar{\mathcal G}_2(t),
v_2-\bar u_2(t)
\right\rangle
&\le0,
&&v_2\in\Pi_2,
\end{aligned}
\right.
\end{equation*}
for $d\mathbb P\otimes dt$-a.e. $(t,\omega)$.

Thus,
$(\bar X,\bar Y,\bar Z,\bar{\mathcal K},\bar u)$
satisfies the FBSDEJ and the cone-type variational
inequalities in Theorem \ref{THE265}. Since Assumption
\ref{A33} implies the UCC condition, Theorem
\ref{THE265} shows that $\bar u$ is an open-loop saddle
point. Its uniqueness follows from Theorem \ref{26081};
hence,
$
\bar u=u^*,\ 
\bar X=X^*.
$
Therefore,
$$
u^*(t)
=
\Theta^+(t)(X^*(t-))^+
+
\Theta^-(t)(X^*(t-))^-,
$$
which proves the feedback-form representation.

Finally, by \eqref{localized-feedback-value}, then
$
\mathcal J_n(x; u_1^*,u_2^*)
=
\mathbb E\left[
P_1(0)(x^+)^2+P_2(0)(x^-)^2
\right].
$
Since $\tau_n=T$ for all sufficiently large $n$, pathwise,
\[
\Phi_n( X^*(\tau_n))
\to
G(T)| X^*(T)|^2.
\]
Moreover,
$
0\le
\Phi_n( X^*(\tau_n))
\le
K\sup_{0\le t\le T}|X^*(t)|^2,
$
and the running integrand is bounded in absolute value by
$
C\left(
| X^*(t)|^2
+
| u_1^*(t)|^2
+
| u_2^*(t)|^2
\right).
$
Dominated convergence therefore gives
$
\mathcal J_n(x;u_1^*, u_2^*)
\to
J(x;u_1^*,u_2^*).
$
Consequently, (\ref{071206}) hold.
\end{proof}

\section{Solvability of the IESREJs} \label{sec4}
\setcounter{equation}{0}
\renewcommand{\theequation}{\thesection.\arabic{equation}}
It is worth emphasizing that the IESREJs \eqref{csre1} are indefinite. Even under the UCC condition, their solvability is difficult to establish directly. Indeed, the cone constraints replace the classical stationarity conditions with variational inequalities and therefore prevent a direct application of the four-step scheme to derive the Riccati system. As a result, the method of Sun et al. \cite{MR4254486} for proving the solvability of indefinite SREs under uniform convexity does not apply directly to the present problem. Under additional structural conditions, we instead prove the solvability of \eqref{csre1} by combining a double-truncation approximation with the multidimensional BSDEJ comparison results of Hu et al. \cite{HSX2025,HSX20251}. The details are presented below.

 We focus on the solvability of the system of I-ESREJs (\ref {csre1}) with  coefficient condition
\begin{equation}\label{071126}
  F_2(t,z)=0,\quad dt\otimes d\mathbb P\otimes\nu(dz)\text{-a.e.},
\qquad
R_{12}(t)=0,\quad D_1(t)D_2(t)^\top=0,
\quad dt\otimes d\mathbb P\text{-a.e.}
\end{equation}
 Under (\ref{071126}) constraints, the  $H_k$ admits the decomposition 
 $$H_k(\omega,t,v_1,v_2, P_1, P_2, \Lambda_k, \Gamma_1,\Gamma_2)=\underbar{H}_{k}(\omega,t,v_1, P_1, P_2, \Lambda_k, \Gamma_1,\Gamma_2)+\bar{H}_{k}(\omega,t, v_2,P_k,  \Lambda_k), \ k=1,2.$$
In this section, we choose the constants $\bar{\delta} $  and $K$ in  Assumptions \ref{A33} as follows
  \begin{equation}\label{2622}
     \bar{\delta} \stackrel{\triangle}{=}(\bar{c}+1)^2e^{4\bar{c}T}(e^{2\bar{c}T}-1),\ \  K \stackrel{\triangle}{=}(\bar{c}+1)e^{2\bar{c}T},  
  \end{equation}
  We  will prove $K$ to be an upper bound of  $P_k(\cdot), k=1,2.$
  \begin{theorem}\label{06031}
Let \eqref{071126} and Assumptions \ref{A1} and \ref{A33}
hold, where the constants $\bar\delta$ and $K$ in Assumption \ref{A33} are given by \eqref{2622}. Then the system of IESREJs \eqref{csre1} admits a solution
$
(P_k,\Lambda_k,\Gamma_k)_{k=1,2}
$
such that $$
\mathbb P\left(
0<P_k(t)\le K,\quad 0\le t\le T
\right)=1,
\qquad k=1,2.
$$
\end{theorem}
\begin{proof}
For $n\in\mathbb N^+$, define
\[
\begin{aligned}
\underline H_k^n
(\omega,t,P_1,P_2,\Lambda_k,\Gamma_1,\Gamma_2)
:={}&
\inf_{\substack{v_1\in\Pi_1\\ |v_1|\le n}}
\underline H_k
(\omega,t,v_1,P_1,P_2,\Lambda_k,\Gamma_1,\Gamma_2),
\\
\bar H_k^n
(\omega,t,P_k,\Lambda_k)
:={}&
\sup_{\substack{v_2\in\Pi_2\\ |v_2|\le n}}
\bar H_k
(\omega,t,v_2,P_k,\Lambda_k),
\qquad k=1,2.
\end{aligned}
\]
For each $(n,\bar n)\in\NN^+\times\NN^+$,  the
following  BSDEJ:
\begin{equation}\label{csre1224}
\left\{\begin{array}{ccl}
dP_1^{n,\bar{n}}(t)&=&-\bigg\{\big(2A(t)+C^2(t)\big)P_1^{n,\bar{n}}(t-)  +2C(t)\Lambda_1^{n,\bar{n}}(t)
 +Q(t)+\bar{H}^{\bar{n}}_{1}(\omega,t, P_1^{n,\bar{n}},\Lambda_1^{n,\bar{n}})  \\
 &&\ \ +\underline{H}_{1}^n(\omega,t, P_1^{n,\bar{n}}, P_2^{n,\bar{n}},\Lambda_1^{n,\bar{n}}, \Gamma_1^{n,\bar{n}}, \Gamma_2^{n,\bar{n}}) \bigg\}dt +\Lambda_1^{n,\bar{n}}(t)dW(t)+\int_{U_0}\Gamma_1^{n,\bar{n}}(t,z)\tilde{N}(dt,dz),\\
 dP_2^{n,\bar{n}}(t)&=&-\bigg\{\big(2A(t)+C^2(t) \big)P_2^{n,\bar{n}}(t-)+2C(t) \Lambda_2^{n,\bar{n}}(t)
 +Q(t)+\bar{H}^{\bar{n}}_{2}(\omega,t, P_2^{n,\bar{n}},\Lambda_2^{n,\bar{n}})
 \\
 &&\ \  +\underline{H}_{2}^n(\omega,t, P_1^{n,\bar{n}}, P_2^{n,\bar{n}},\Lambda_2^{n,\bar{n}}, \Gamma_1^{n,\bar{n}}, \Gamma_2^{n,\bar{n}}) \bigg\}dt +\Lambda_2^{n,\bar{n}}(t)dW(t)+\int_{U_0}\Gamma_2^{n,\bar{n}}(t,z)\tilde{N}(dt,dz),\\
P_1^{n,\bar{n}}(T)&=&G(T),\ P_2^{n,\bar{n}}(T)=G(T),
\end{array}\right.\end{equation}
 is a two-dimensional coupled BSDEJ with a Lipschitz generator. According to Lemma 2.4 in \cite{TL1994}, it admits a unique solution $\big(P_k^{n,\bar{n}}(\cdot),\Lambda_k^{n,\bar{n}}(\cdot), \Gamma_k^{n,\bar{n}} (\cdot,\cdot)\big)\in \cS^2_{\mathbb{F}}(0,T; \mathbb{R})\times L_{\mathbb{F}}^2(0,T; \mathbb{R})\times\cL_{\cP}^{2,\nu}(0,T; \mathbb{R}),\  \  \text{for}\  k=1,2$. 
 
We verify the conditions of
\cite[Theorem 2.2]{HSX2025}. For the first component,
the boundedness of the controls in the truncated Hamiltonians
gives
\begin{align}
&\underline H_1^n
(\omega,t,P_1,P_2,\Lambda_1,\Gamma_1',\Gamma_2)
-
\underline H_1^n
(\omega,t,P_1',P_2',\Lambda_1',\Gamma_1',\Gamma_2')+
\bar H_1^{\bar n}(\omega,t,P_1,\Lambda_1)
-
\bar H_1^{\bar n}(\omega,t,P_1',\Lambda_1')
\notag\\
&\le
C_{n,\bar n}
\Bigg[
|P_1-P_1'|
+
(P_2-P_2')^+
+
|\Lambda_1-\Lambda_1'|
+
\int_{U_0}
\big(
P_2+\Gamma_2(z)
-P_2'-\Gamma_2'(z)
\big)^+\nu(dz)
\Bigg].
\label{eq:comparison-cross-condition}
\end{align}
Keeping all the other variables fixed, we also have
\begin{align*}
&\underline H_1^n
(\omega,t,P_1,P_2,\Lambda_1,\Gamma_1,\Gamma_2)
-
\underline H_1^n
(\omega,t,P_1,P_2,\Lambda_1,\Gamma_1',\Gamma_2)
\notag\\
&\le
\sup_{\substack{v_1\in\Pi_1\\|v_1|\le n}}
\int_{U_0}
\big(
\Gamma_1(z)-\Gamma_1'(z)
\big)
\Big[
\big(
(1+E(z)+F_1(z)^\top v_1)^+
\big)^2-1
\Big]\nu(dz)
\notag\\
&\le
C_n
\int_{U_0}
\big(
\Gamma_1(z)-\Gamma_1'(z)
\big)^+\nu(dz)
+
\int_{U_0}
|\Gamma_1(z)-\Gamma_1'(z)|\nu(dz).
\label{eq:comparison-own-jump-condition}
\end{align*}
The corresponding estimates hold for the second component. After adding the linear terms
$(2A+C^2)P_k+2C\Lambda_k+Q,$
the generators satisfy conditions {\rm H(2)} and {\rm H(3)} of \cite[Theorem 2.2]{HSX2025}. Their values at the origin are square integrable, and both generators are globally Lipschitz. Since
$
\underline H_k^{n+1}
\le
\underline H_k^n
$, $
\bar{H}_k^{\bar{n}}
\le
\bar{ H}_k^{\bar{n}+1}
$. Therefore,
\cite[Theorem 2.2]{HSX2025} yields
\begin{equation}\label{eq:monotonicity-inner}
P_k^{n+1,\bar n}(t)
\le
P_k^{n,\bar n}(t) \ \text{and}\ P_k^{n,\bar {n}}(t)
\le
P_k^{n,\bar {n}+1}(t),
\qquad
0\le t\le T,\quad k=1,2.
\end{equation}

Since ${\bf 0}\in\Pi_1$, we have
$
\underline H_k^n
(\omega,t,P_1,P_2,\Lambda_k,\Gamma_1,\Gamma_2)
\le
\Phi_k(P_1,P_2,\Gamma_1,\Gamma_2),
 k=1,2,
$
where
\[
\begin{aligned}
\Phi_1(P_1,P_2,\Gamma_1,\Gamma_2)
:={}&
\int_{U_0}
\Big[
(P_1+\Gamma_1(z))
\big(
((1+E(z))^+)^2-1
\big)
\\
&\qquad
-2P_1E(z)
+
(P_2+\Gamma_2(z))
((1+E(z))^-)^2
\Big]\nu(dz),
\\
\Phi_2(P_1,P_2,\Gamma_1,\Gamma_2)
:={}&
\int_{U_0}
\Big[
(P_2+\Gamma_2(z))
\big(
((-1-E(z))^-)^2-1
\big)
\\
&\qquad
-2P_2E(z)
+
(P_1+\Gamma_1(z))
((-1-E(z))^+)^2
\Big]\nu(dz).
\end{aligned}
\]
Moreover, since ${\bf 0}\in\Pi_2$ and
$\bar H_k(\omega,t,{\bf 0},P_k,\Lambda_k)=0$, we have 
$
\bar H_k^{\bar n}
\ge0, k=1,2.
$

Note that
$
(\underline P,\underline\Lambda,\underline\Gamma)
=
\big(
\underline\delta e^{-\underline c_1(T-t)},0,0
\big)
$
is the unique solution to the one-dimensional BSDEJ
\begin{equation*}\label{csre1227}
\left\{\begin{array}{ccl}
d\underline{P}(t)&=&-(-\underline{c}_1\underline{P}(t))dt +\underline{\Lambda}(t)dW(t)+\int_{U_0}\underline{\Gamma}(t,z)\tilde{N}(dt,dz),\\
\underline{P}(T)&=&\underline{\delta},\\
\end{array}\right.\end{equation*}
where $\underline c_1>0$ is chosen such that
\begin{align}
&2A+C^2
+
\int_{U_0}E(z)^2\nu(dz)
-
\underline\delta^{-1}
\left|
B_1^\top 
+D_1^\top C
+
\int_{U_0}
E(z)F_1^\top(z)\nu(dz)
\right|^2
\ge
-\underline c_1.
\label{eq:lower-barrier-coefficient}
\end{align}
Using $(r^+)^2+(r^-)^2=r^2, r\in\mathbb R,$
the structural condition \eqref{071126}, and the definition
of $\underline H_k^n$, we obtain
\begin{align*}
&\underline H_k^n
\big(
\omega,t,
\underline P,
\underline P,
0,0,0
\big)
\\
&\ge
\underline P(t)
\int_{U_0}E(t,z)^2\nu(dz)
+
\inf_{v_1\in\mathbb R^{m_1}}
v_1^\top R_{11}(t)v_1
\\
&\quad+
\underline P(t)
\inf_{v_1\in\mathbb R^{m_1}}
\Bigg\{
v_1^\top
\left[
D_1(t)D_1(t)^\top
+
\int_{U_0}
F_1(t,z)F_1(t,z)^\top\nu(dz)
\right]v_1
\\
&\hspace{39mm}
\mathbin{\pm}
2
\left[
B_1^\top(t)
+D_1(t)^\top C(t)
+
\int_{U_0}
E(t,z)F_1^\top(t,z)\nu(dz)
\right] v_1
\Bigg\},
\end{align*}
where the plus sign corresponds to $k=1$ and the minus sign
corresponds to $k=2$.  By Assumption \ref{A33} and completing the square, we obtain
\begin{align}
&\underline H_k^n
\big(
\omega,t,
\underline P,
\underline P,
0,0,0
\big)
\ge
\underline P
\Bigg[
\int_{U_0}E(t,z)^2\nu(dz)\notag\\
&
-
\underline\delta^{-1}
\left|
B_1^\top(t)
+D_1(t)^\top C(t)
+
\int_{U_0}
E(t,z)F_1^\top(t,z)\nu(dz)
\right|^2
\Bigg], \ \text{k=1,2}
\label{eq:lower-Hamiltonian-estimate}
\end{align}
Since $Q\ge0$ and $\bar H_k^{\bar n}\ge0$,
\eqref{eq:lower-barrier-coefficient} and
\eqref{eq:lower-Hamiltonian-estimate} imply
\begin{align}
&(2A(t)+C(t)^2)\underline P(t)
+Q(t)
+
\underline H_k^n
\big(
\omega,t,
\underline P,
\underline P,
0,0,0
\big)+
\bar H_k^{\bar n}
\big(
\omega,t,
\underline P,
0
\big)\ge
-\underline c_1\underline P(t),
\  k=1,2.
\label{eq:lower-generator-comparison}
\end{align}

Consider the two-dimensional barrier process
$ (\underline P,\underline P,0,0,0,0).$
Its generator is linear and globally Lipschitz. Moreover,
$
\underline P(T)
=
\underline\delta
\le
G(T).
$
Therefore, \eqref{eq:lower-generator-comparison} and
\cite[Theorem 2.2]{HSX2025} give
\begin{equation}\label{eq:lower-bound-truncated}
P_k^{n,\bar n}(t)
\ge
\underline P(t)
\ge
\underline\delta e^{-\underline c_1T}
=:
\varepsilon_0,
\qquad
0\le t\le T,\quad k=1,2.
\end{equation}

On the other hand, note that, $(\bar{P},0, 0)=((\bar{c}+\frac{\bar{\delta}+K^2}{\bar{\delta}})e^{2\bar{c}(T-t)}-\frac{\bar{\delta}+K^2}{\bar{\delta}},0,0)$
is the unique solution to the one-dimensional BSDEJ
\begin{equation*}
\left\{\begin{array}{ccl}
d\bar{P}(t)&=&-2\{\bar{c}\bar{P}+\bar{c}+\bar{\delta}^{-1}\bar{c}{K}^2\}dt +\bar{\Lambda} dW(t)+\int_{U_0}\bar{\Gamma}(z)\tilde{N}(dt,dz),\\
\bar{P}(T)&=&\bar{c}.\\
\end{array}\right.\end{equation*}
By the choice of $\bar\delta$ and $K$ in \eqref{2622}, we have 
$
0<\bar P(t)\le K,
\qquad 0\le t\le T.
$

Since the Brownian and jump components of $\bar P$ vanish,
\begin{align*}
\bar H_k^{\bar n}(\omega,t,\bar P,0)
=
\sup_{\substack{v_2\in\Pi_2\\ |v_2|\le\bar n}}
\bar H_k(\omega,t,v_2,\bar P,0)
\le
\sup_{v_2\in\mathbb R^{m_2}}
\bar H_k(\omega,t,v_2,\bar P,0)
\le
2\bar\delta^{-1}\bar cK^2.
\end{align*}
Furthermore,
\[
\Phi_k(\bar P,\bar P,0,0)
=
\bar P
\int_{U_0}E(z)^2\nu(dz),
\qquad k=1,2.
\]
Consequently, by (\ref{26051}), we have
\begin{align*}
&(2A(t)+C(t)^2)\bar P(t)
+Q(t)
+
\underline H_k^n
\big(
\omega,t,
\bar P,
\bar P,
0,0,0
\big)
\notag\\
&\quad+
\bar H_k^{\bar n}
\big(
\omega,t,
\bar P,
0
\big)
\le
2\bar c\bar P(t)
+
2\bar c
+
2\bar\delta^{-1}\bar cK^2,
\qquad k=1,2.
\label{eq:upper-generator-comparison}
\end{align*}
Since
$
G(T)\le\bar c=\bar P(T),
$
Theorem 2.2 in \cite{HSX2025} gives
\begin{equation}\label{eq:upper-bound-truncated}
P_k^{n,\bar n}(t)
\le
\bar P(t)
\le K,\ 
 t\in[0,T],\quad k=1,2.
\end{equation}
Combining \eqref{eq:lower-bound-truncated} and
\eqref{eq:upper-bound-truncated}, we obtain
\begin{equation}\label{eq:truncated-path-bounds}
\varepsilon_0
\le
P_k^{n,\bar n}(t)
\le K,
\qquad
0\le t\le T,\quad k=1,2.
\end{equation}

We next derive the corresponding post-jump bounds. Define
\[
\begin{aligned}
A_k^{n,\bar n}
:=
\big\{
(\omega,t,z):
&P_k^{n,\bar n}(t-)
+\Gamma_k^{n,\bar n}(t,z)
\notin[\varepsilon_0,K]
\big\}.
\end{aligned}
\]
At every atom $(T_j,Z_j)$ of the Poisson random measure,
$
P_k^{n,\bar n}(T_j)
=
P_k^{n,\bar n}(T_j-)
+
\Gamma_k^{n,\bar n}(T_j,Z_j).
$
Hence, by the compensation formula and
\eqref{eq:truncated-path-bounds},
\[
\begin{aligned}
\mathbb E
\int_0^T\int_{U_0}
\mathbf 1_{A_k^{n,\bar n}}
\nu(dz)dt=
\mathbb E
\int_0^T\int_{U_0}
\mathbf 1_{A_k^{n,\bar n}}
N(dt,dz)=
\mathbb E
\sum_{T_j\le T}
\mathbf 1_{\{
P_k^{n,\bar n}(T_j)
\notin[\varepsilon_0,K]
\}}
=0.
\end{aligned}
\]
Therefore,
\begin{equation}\label{eq:truncated-postjump-bounds}
\varepsilon_0
\le
P_k^{n,\bar n}(t-)
+\Gamma_k^{n,\bar n}(t,z)
\le K,
\qquad k=1,2,
\end{equation}
$d\mathbb P\otimes dt\otimes\nu(dz)$-a.e. Since
$
\varepsilon_0
\le
P_k^{n,\bar n}(t-)
\le K,
$
we also have
\begin{equation}\label{eq:truncated-Gamma-bound}
|\Gamma_k^{n,\bar n}(t,z)|
\le
K-\varepsilon_0,
\qquad k=1,2.
\end{equation}

Fix $\bar n\in\mathbb N^+$. By
\eqref{eq:monotonicity-inner} and
\eqref{eq:truncated-path-bounds}, the limit
$
P_k^{\bar n}(t)
:=
\lim_{n\to\infty}
P_k^{n,\bar n}(t)
$
is well defined and satisfies
$
\varepsilon_0\le P_k^{\bar n}(t)\le K.
$
Moreover,
\begin{equation}\label{eq:first-P-limit}
\lim_{n\to\infty}
\sum_{k=1}^2
\mathbb E\int_0^T
|P_k^{n,\bar n}(t)-P_k^{\bar n}(t)|^2dt
=0.
\end{equation}
The two-sided Hamiltonian estimates established in (\ref{071302})-(\ref{071303})
and \eqref{eq:truncated-postjump-bounds} imply
$
-C\big(
1+|\Lambda_k^{n,\bar n}|^2
\big)
\le
\underline H_k^n
\le C,
$
while, for fixed $\bar n$,
$
0\le
\bar H_k^{\bar n}
\le
C_{\bar n}
\big(
1+|\Lambda_k^{n,\bar n}|
\big).
$
Itô's formula applied to $|P_k^{n,\bar n}|^2$ therefore
gives
\begin{equation*}\label{eq:first-uniform-estimate}
\sup_{n\ge1}
\sum_{k=1}^2
\mathbb E\int_0^T
\left[
|\Lambda_k^{n,\bar n}(t)|^2
+
\int_{U_0}
|\Gamma_k^{n,\bar n}(t,z)|^2\nu(dz)
\right]dt
<\infty.
\end{equation*}

For $m>n$, both
$
P_k^{n,\bar n}-P_k^{m,\bar n}
$
and its post-jump value are nonnegative. Hence the monotone
stability argument in the proof of
\cite[Theorem 3.1]{HSX20251}, in particular
\cite[Lemma 3.1 and Appendix B]{HSX20251}, applies. The additional term
$\bar H_k^{\bar n}$ is globally Lipschitz and satisfies, for
every $\eta>0$,
\[
\begin{aligned}
\Psi_\kappa'
\big(
P_k^{n,\bar n}-P_k^{m,\bar n}
\big)
\left|
\bar H_k^{\bar n}
(P_k^{n,\bar n},\Lambda_k^{n,\bar n})
-
\bar H_k^{\bar n}
(P_k^{m,\bar n},\Lambda_k^{m,\bar n})
\right|
\le
\eta
|\Lambda_k^{n,\bar n}-\Lambda_k^{m,\bar n}|^2
+
C_{\eta,\bar n}
|P_k^{n,\bar n}-P_k^{m,\bar n}|^2,
\end{aligned}
\]
where
\[
\Psi_\kappa(x)
:=
\frac{e^{\kappa x}-\kappa x-1}{\kappa},
\qquad x\in\mathbb R.
\]
It is therefore absorbed by Young's inequality in the
exponential-difference estimate. Consequently, there exist
processes $\Lambda_k^{\bar n}$ and $\Gamma_k^{\bar n}$ such
that
\begin{equation}\label{eq:first-strong-limit}
\begin{aligned}
\lim_{n\to\infty}
\sum_{k=1}^2
\Bigg[
\mathbb E\int_0^T
|\Lambda_k^{n,\bar n}(t)-\Lambda_k^{\bar n}(t)|^2dt
+
\mathbb E\int_0^T\int_{U_0}
|\Gamma_k^{n,\bar n}(t,z)-\Gamma_k^{\bar n}(t,z)|^2
\nu(dz)dt
\Bigg]
=0.
\end{aligned}
\end{equation}

Let $\widehat v_{1,k}^{\,n,\bar n}$ denote an untruncated minimizer corresponding to the approximate parameters. By the growth estimate $
|\widehat v_{1,k}^{\,n,\bar n}|
\le
C\big(
1+|\Lambda_k^{n,\bar n}|
\big),
$ established in Section 4, with $C$ independent of $n$. Thus, along an almost
everywhere convergent subsequence, the inner truncation is
eventually inactive. The continuity of the Hamiltonians,
\eqref{eq:first-P-limit}, and
\eqref{eq:first-strong-limit} give pointwise convergence of the Hamiltonian terms. Moreover,
\[
\left|
\underline H_k^n
\big(\omega,t,
P_1^{n,\bar n},
P_2^{n,\bar n},
\Lambda_k^{n,\bar n},
\Gamma_1^{n,\bar n},
\Gamma_2^{n,\bar n}
\big)
\right|
\le
C\big(
1+|\Lambda_k^{n,\bar n}|^2
\big).
\]
By \eqref{eq:first-strong-limit}, we have 
$
|\Lambda_k^{n,\bar n}|^2
\longrightarrow
|\Lambda_k^{\bar n}|^2
\quad\text{in }
L^1(d\mathbb P\otimes dt)
$(here and below, $f_n\to f$ in
$L^1(d\mathbb P\otimes dt)$ means
$\mathbb E\int_0^T|f_n(t)-f(t)|dt\to0$).
Hence, the family
$
\left\{
\underline H_k^n
\right\}_{n\ge1}
$
is uniformly integrable.

Moreover, the strong convergence in
\eqref{eq:first-P-limit} and
\eqref{eq:first-strong-limit} implies convergence in measure of the corresponding parameters. Every subsequence
therefore admits a further subsequence along which the parameters converge almost everywhere. Along such a further subsequence, the growth estimate for the untruncated minimizers implies that the inner truncation is eventually inactive, and hence
\[
\begin{aligned}
&
\underline H_k^n
\big(\omega,t,
P_1^{n,\bar n},
P_2^{n,\bar n},
\Lambda_k^{n,\bar n},
\Gamma_1^{n,\bar n},
\Gamma_2^{n,\bar n}
\big)
\longrightarrow
\underline H_k^*
\big(\omega,t,
P_1^{\bar n},
P_2^{\bar n},
\Lambda_k^{\bar n},
\Gamma_1^{\bar n},
\Gamma_2^{\bar n}
\big)
\end{aligned}
\]
$d\mathbb P\otimes dt$-a.e. By the subsequence principle, the inner Hamiltonians converge in measure along the whole sequence. Consequently, Vitali's theorem gives
\begin{align}
\lim_{n\to\infty}
\mathbb E\int_0^T
\Big|
&
\underline H_k^n
\big(\omega,t,
P_1^{n,\bar n},
P_2^{n,\bar n},
\Lambda_k^{n,\bar n},
\Gamma_1^{n,\bar n},
\Gamma_2^{n,\bar n}
\big)
-
\underline H_k^*
\big(\omega,t,
P_1^{\bar n},
P_2^{\bar n},
\Lambda_k^{\bar n},
\Gamma_1^{\bar n},
\Gamma_2^{\bar n}
\big)
\Big|dt
=0.
\label{eq:first-inner-Hamiltonian-limit}
\end{align}
Moreover, since $\bar H_k^{\bar n}$ is globally Lipschitz
for fixed $\bar n$,
\[
\begin{aligned}
&\mathbb E\int_0^T
\left|
\bar H_k^{\bar n}
\big(\omega,t,
P_k^{n,\bar n},\Lambda_k^{n,\bar n}
\big)
-
\bar H_k^{\bar n}
\big(\omega,t,
P_k^{\bar n},\Lambda_k^{\bar n}
\big)
\right|dt
\le
C_{\bar n}
\mathbb E\int_0^T
\left(
|P_k^{n,\bar n}-P_k^{\bar n}|
+
|\Lambda_k^{n,\bar n}-\Lambda_k^{\bar n}|
\right)dt
\longrightarrow0.
\end{aligned}
\]
Combining\eqref{eq:first-inner-Hamiltonian-limit} with the preceding convergence of the outer Hamiltonians,
we obtain
\begin{align}
\underline H_k^n
\big(\omega,t,
P_1^{n,\bar n},P_2^{n,\bar n},
\Lambda_k^{n,\bar n},
\Gamma_1^{n,\bar n},\Gamma_2^{n,\bar n}
\big)
+
\bar H_k^{\bar n}
\big(\omega,t,
P_k^{n,\bar n},\Lambda_k^{n,\bar n}
\big)
&\longrightarrow
\underline H_k^*
\big(\omega,t,
P_1^{\bar n},P_2^{\bar n},
\Lambda_k^{\bar n},
\Gamma_1^{\bar n},\Gamma_2^{\bar n}
\big)\nonumber\\
&\ \ \ +
\bar H_k^{\bar n}
\big(\omega,t,
P_k^{\bar n},\Lambda_k^{\bar n}
\big)
\label{eq:first-Hamiltonian-limit}
\end{align}
in $L^1(d\mathbb P\otimes dt)$. By \eqref{eq:first-P-limit},
\eqref{eq:first-strong-limit}, and
\eqref{eq:first-Hamiltonian-limit}, the drift terms in
\eqref{csre1224} converge in
$L^1(d\mathbb P\otimes dt)$. Moreover, the
BDG inequalities and
\eqref{eq:first-strong-limit} imply that the corresponding Brownian and compensated Poisson stochastic integrals
converge in $\mathcal S_{\mathbb F}^2$. Hence, passing to
the limit in the integral form of \eqref{csre1224}, we
obtain c\`adl\`ag versions of $P_1^{\bar n}$ and
$P_2^{\bar n}$ and conclude that
$(P^{\bar n},\Lambda^{\bar n},\Gamma^{\bar n})$ satisfies
\begin{equation}\label{eq:intermediate-bsdej}
\left\{
\begin{aligned}
dP_k^{\bar n}(t)
={}&-
\Big\{
(2A+C^2)P_k^{\bar n}(t-)
+2C\Lambda_k^{\bar n}(t)
+Q
+
\underline H_k^*
\big(\omega,t,
P_1^{\bar n}(t-),P_2^{\bar n}(t-),
\Lambda_k^{\bar n}(t),
\Gamma_1^{\bar n}(t,\cdot),
\Gamma_2^{\bar n}(t,\cdot)
\big)
\\
&\quad+
\bar H_k^{\bar n}
\big(\omega,t,
P_k^{\bar n}(t-),
\Lambda_k^{\bar n}(t)
\big)
\Big\}dt
+
\Lambda_k^{\bar n}(t)dW(t)
+
\int_{U_0}
\Gamma_k^{\bar n}(t,z)\widetilde N(dt,dz),
\\
P_k^{\bar n}(T)
={}&G(T),
\qquad k=1,2.
\end{aligned}
\right.
\end{equation}
By \eqref{eq:first-P-limit} and
\eqref{eq:first-strong-limit}, after passing to a
subsequence, we have $
P_k^{n,\bar n}(t)
\to
P_k^{\bar n}(t)
$ and
$
\Gamma_k^{n,\bar n}(t,z)
\to
\Gamma_k^{\bar n}(t,z)
$
$d\mathbb P\otimes dt\otimes\nu(dz)$-a.e.
Since a c\`adl\`ag process and its left-limit process agree
$dt$-a.e.,
\[
P_k^{n,\bar n}(t-)
+
\Gamma_k^{n,\bar n}(t,z)
\longrightarrow
P_k^{\bar n}(t-)
+
\Gamma_k^{\bar n}(t,z)
\]
$d\mathbb P\otimes dt\otimes\nu(dz)$-a.e. Passing to the limit in \eqref{eq:truncated-postjump-bounds}, we obtain
\begin{equation}\label{eq:intermediate-postjump-bound}
\varepsilon_0
\le
P_k^{\bar n}(t-)+\Gamma_k^{\bar n}(t,z)
\le K,  \ d\mathbb P\otimes dt\otimes\nu(dz)\text{-a.e.}
\end{equation}

Letting $n\to\infty$ in the second inequality of \eqref{eq:monotonicity-inner}, we obtain
$
P_k^{\bar n}(t)
\le
P_k^{\bar n+1}(t).
$
Hence,
$
P_k(t)
:=
\lim_{\bar n\to\infty}
P_k^{\bar n}(t)
$
is well defined and
\begin{equation}\label{eq:second-P-limit}
\varepsilon_0\le P_k(t)\le K,
\qquad
P_k^{\bar n}\longrightarrow P_k
\quad\text{in }L^2(d\mathbb P\otimes dt).
\end{equation}

The two-sided Hamiltonian estimates imply, uniformly in
$\bar n$,
\[
\begin{aligned}
&(2A+C^2)P_k^{\bar n}
+2C\Lambda_k^{\bar n}
+Q
+
\underline H_k^*
\big(\omega,t,
P_1^{\bar n},P_2^{\bar n},
\Lambda_k^{\bar n},
\Gamma_1^{\bar n},\Gamma_2^{\bar n}
\big)+
\bar H_k^{\bar n}
\big(\omega,t,
P_k^{\bar n},\Lambda_k^{\bar n}
\big)
\le
C\big(
1+|\Lambda_k^{\bar n}|^2
\big).
\end{aligned}
\]
Together with
$
|\Gamma_k^{\bar n}|
\le K-\varepsilon_0,
$
Itô's formula applied to $e^{\beta P_k^{\bar n}}$, with
$\beta>0$ sufficiently large, gives
\begin{equation*}\label{eq:second-uniform-estimate}
\sup_{\bar n\ge1}
\sum_{k=1}^2
\mathbb E\int_0^T
\left[
|\Lambda_k^{\bar n}(t)|^2
+
\int_{U_0}
|\Gamma_k^{\bar n}(t,z)|^2\nu(dz)
\right]dt
<\infty.
\end{equation*}
Here we used
$
e^{\beta x}-1-\beta x
\ge
c_\beta x^2,
|x|\le K-\varepsilon_0,
$ for some $c_\beta>0$.

By weak compactness, along a subsequence,
\[
\Lambda_k^{\bar n}
\rightharpoonup
\Lambda_k
\quad\text{in }\mathcal L_{\mathbb F}^2,
\qquad
\Gamma_k^{\bar n}
\rightharpoonup
\Gamma_k
\quad\text{in }\mathcal L_{\mathcal P}^{2,\nu}.
\]
For $\bar m>\bar n$, the processes $P_k^{\bar m}-P_k^{\bar n}$and their post-jump values are nonnegative.
More precisely, the compensation argument and
\eqref{eq:intermediate-postjump-bound} give
\[
0
\le
P_k^{\bar m}(t-)-P_k^{\bar n}(t-)
+
\Gamma_k^{\bar m}(t,z)-\Gamma_k^{\bar n}(t,z)
\le
K-\varepsilon_0
\]
$d\mathbb P\otimes dt\otimes\nu(dz)$-a.e.

Set
\[
\begin{aligned}
\mathcal R_k^{\bar r}
:={}&
\underline H_k^*
\big(\omega,t,
P_1^{\bar r},P_2^{\bar r},
\Lambda_k^{\bar r},
\Gamma_1^{\bar r},\Gamma_2^{\bar r}
\big)
+
\bar H_k^{\bar r}
\big(\omega,t,
P_k^{\bar r},\Lambda_k^{\bar r}
\big),
\qquad
\bar r\in\mathbb N^+.
\end{aligned}
\]
By the uniform quadratic estimate, $
|\mathcal R_k^{\bar r}|
\le
C\big(
1+|\Lambda_k^{\bar r}|^2
\big).
$
Consequently,
$$
\mathcal R_k^{\bar m}
-
\mathcal R_k^{\bar n}
\le
C\left(
1+|\Lambda_k^{\bar m}|^2
+|\Lambda_k^{\bar n}|^2
\right)
\le
C\left(
1+
|\Lambda_k^{\bar m}-\Lambda_k^{\bar n}|^2
+
|\Lambda_k^{\bar n}|^2
\right),
$$
where $C$ is independent of $\bar m$ and $\bar n$.
For $\bar m>\bar n$, set
$$\Delta P_k^{\bar m,\bar n}:=P_k^{\bar m}-P_k^{\bar n},\ \Delta\Lambda_k^{\bar m,\bar n}:=\Lambda_k^{\bar m}-\Lambda_k^{\bar n} \text{and} \Delta\Gamma_k^{\bar m,\bar n}:=\Gamma_k^{\bar m}-\Gamma_k^{\bar n}.$$
By monotonicity,
$
\Delta P_k^{\bar m,\bar n}(t-)\ge0.
$
Moreover, the compensation argument and
\eqref{eq:intermediate-postjump-bound} give
\[
0
\le
\Delta P_k^{\bar m,\bar n}(t-)
+
\Delta\Gamma_k^{\bar m,\bar n}(t,z)
\le
K-\varepsilon_0 \ \text{for}\  d\mathbb P\otimes dt\otimes\nu(dz) \text{-a.e}.
\]

For $\kappa>0$, define
$
\Psi_\kappa(x)
:=
\frac{e^{\kappa x}-\kappa x-1}{\kappa},
x\in\mathbb R.
$
Since both
$
\Delta P_k^{\bar m,\bar n}(t-)
\ \text{and}\
\Delta P_k^{\bar m,\bar n}(t-)
+
\Delta\Gamma_k^{\bar m,\bar n}(t,z)
$
belong to $[0,K-\varepsilon_0]$, there exists a constant
$d_\kappa>0$ such that
\[
\Psi_\kappa(x+y)
-
\Psi_\kappa(x)
-
\Psi_\kappa'(x)y
\ge
d_\kappa |y|^2
\]
whenever
$
x\in[0,K-\varepsilon_0],\
x+y\in[0,K-\varepsilon_0].
$

Applying It\^o's formula with jumps to
$\Psi_\kappa(\Delta P_k^{\bar m,\bar n})$ on $[0,T]$,
taking expectations, using the preceding jump coercivity estimate and the bound for
$\mathcal R_k^{\bar m}-\mathcal R_k^{\bar n}$, and summing
over $k=1,2$, we obtain
\begin{align}
&\mathbb E
\sum_{k=1}^2
\Psi_\kappa
\big(
\Delta P_k^{\bar m,\bar n}(0)
\big)
+
\sum_{k=1}^2
\mathbb E\int_0^T
\left[
\frac12
\Psi_\kappa''
\big(
\Delta P_k^{\bar m,\bar n}(t)
\big)
-
C
\Psi_\kappa'
\big(
\Delta P_k^{\bar m,\bar n}(t)
\big)
\right]
\left|
\Delta\Lambda_k^{\bar m,\bar n}(t)
\right|^2dt
\notag\\
&\quad+
d_\kappa
\sum_{k=1}^2
\mathbb E\int_0^T\int_{U_0}
\left|
\Delta\Gamma_k^{\bar m,\bar n}(t,z)
\right|^2
\nu(dz)dt
\notag\\
&\le
C
\sum_{k=1}^2
\mathbb E\int_0^T
\Psi_\kappa'
\big(
\Delta P_k^{\bar m,\bar n}(t)
\big)
\left[
1+
\left|
\Delta\Lambda_k^{\bar m,\bar n}(t)
\right|
+
|\Lambda_k^{\bar n}(t)|^2
\right]dt.
\label{eq:second-exponential-difference-estimate}
\end{align}

Set
$
a_\kappa(x)
:=
\frac12\Psi_\kappa''(x)
-
C\Psi_\kappa'(x).
$
Since $
\Psi_\kappa'(x)=e^{\kappa x}-1,
\
\Psi_\kappa''(x)=\kappa e^{\kappa x},
$
we have
$
a_\kappa(x)
=
\frac{\kappa}{2}
+
\left(
\frac{\kappa}{2}-C
\right)
\Psi_\kappa'(x).
$
Choose $\kappa>2C$, then
$
a_\kappa(x)\ge\frac{\kappa}{2}>0,\
0\le x\le K-\varepsilon_0.
$
Applying Young's inequality to the linear term on the
right-hand side of
\eqref{eq:second-exponential-difference-estimate}, we obtain
\[
\begin{aligned}
&C
\Psi_\kappa'
\big(
\Delta P_k^{\bar m,\bar n}
\big)
\left|
\Delta\Lambda_k^{\bar m,\bar n}
\right|
\le
\frac14
a_\kappa
\big(
\Delta P_k^{\bar m,\bar n}
\big)
\left|
\Delta\Lambda_k^{\bar m,\bar n}
\right|^2
+
C_\kappa
\left|
\Psi_\kappa'
\big(
\Delta P_k^{\bar m,\bar n}
\big)
\right|^2.
\end{aligned}
\]

The first term on the right-hand side is absorbed into the left-hand side of
\eqref{eq:second-exponential-difference-estimate}.
Since $\Psi_\kappa\ge0$, the first term on the left-hand
side may be discarded. We may now let
$\bar m\to\infty$ along the weakly convergent subsequence. The weights generated by
$\Psi_\kappa'$ and $\Psi_\kappa''$ are uniformly bounded and
converge strongly in every finite $L^p$ space. Hence, the
weighted weak lower semicontinuity argument in Appendix B of
\cite{HSX2025} gives
\[
\begin{aligned}
&\sum_{k=1}^2
\mathbb E\int_0^T
a_\kappa
\big(
P_k-P_k^{\bar n}
\big)
\left|
\Lambda_k-\Lambda_k^{\bar n}
\right|^2dt+
d_\kappa
\sum_{k=1}^2
\mathbb E\int_0^T\int_{U_0}
\left|
\Gamma_k-\Gamma_k^{\bar n}
\right|^2
\nu(dz)dt
\\
&\le
C
\sum_{k=1}^2
\mathbb E\int_0^T
\left[
q_k^{\bar n}
+
|q_k^{\bar n}|^2
+
q_k^{\bar n}
|\Lambda_k^{\bar n}|^2
\right]dt.
\end{aligned}
\]
Using
$
|\Lambda_k^{\bar n}|^2
\le
2|\Lambda_k^{\bar n}-\Lambda_k|^2
+
2|\Lambda_k|^2
$
and
$
a_\kappa
\big(
P_k-P_k^{\bar n}
\big)
=
\frac{\kappa}{2}
+
\left(
\frac{\kappa}{2}-C
\right)
q_k^{\bar n},
$
we may choose $\kappa$ sufficiently large and absorb the term
containing
$
q_k^{\bar n}
|\Lambda_k^{\bar n}-\Lambda_k|^2
$
into the left-hand side. Consequently,
\begin{align}
&c_0
\sum_{k=1}^2
\mathbb E\int_0^T
|\Lambda_k^{\bar n}(t)-\Lambda_k(t)|^2dt
+
d_\kappa
\sum_{k=1}^2
\mathbb E\int_0^T\int_{U_0}
|\Gamma_k^{\bar n}(t,z)-\Gamma_k(t,z)|^2
\nu(dz)dt
\notag\\
&\le
C
\sum_{k=1}^2
\mathbb E\int_0^T
\left[
q_k^{\bar n}(t)
+
|q_k^{\bar n}(t)|^2
\right]
\big(
1+|\Lambda_k(t)|^2
\big)dt,
\label{eq:key-second-stability}
\end{align}
where
$
q_k^{\bar n}(t)
:=
\Psi_\kappa'
\big(
P_k(t)-P_k^{\bar n}(t)
\big).
$
By \eqref{eq:second-P-limit},
we have  $q_k^{\bar n}\longrightarrow0, 
d\mathbb P\otimes dt\text{-a.e.},
$
and the sequence is uniformly bounded. The dominated
convergence theorem applied to
\eqref{eq:key-second-stability} gives
\begin{equation}\label{eq:second-strong-limit}
\begin{aligned}
\lim_{\bar n\to\infty}
\sum_{k=1}^2
\Bigg[
&
\mathbb E\int_0^T
|\Lambda_k^{\bar n}(t)-\Lambda_k(t)|^2dt
+
\mathbb E\int_0^T\int_{U_0}
|\Gamma_k^{\bar n}(t,z)-\Gamma_k(t,z)|^2
\nu(dz)dt
\Bigg]
=0.
\end{aligned}
\end{equation}

Although the limit pair $(\Lambda,\Gamma)$ was initially
identified along a weakly convergent subsequence,
\eqref{eq:key-second-stability} holds for every
$\bar n\in\mathbb N^+$. Indeed, for each fixed $\bar n$,
the limit $\bar m\to\infty$ is taken along the weakly
convergent subsequence, while $\bar n$ itself is arbitrary.
Since the right-hand side of
\eqref{eq:key-second-stability} converges to zero as
$\bar n\to\infty$ along the whole sequence,
\eqref{eq:second-strong-limit} holds for the entire sequence,
and not merely for the weakly convergent subsequence.

It remains to pass to the limit in the Hamiltonians. Let an
arbitrary subsequence of
$
\left\{
(P^{\bar n},\Lambda^{\bar n},\Gamma^{\bar n})
\right\}_{\bar n\ge1}
$
be given. By \eqref{eq:second-P-limit} and
\eqref{eq:second-strong-limit}, it admits a further
subsequence, still indexed by $\bar n$, such that, for
$k=1,2$,
$$
P_k^{\bar n}(t)\longrightarrow P_k(t),\ 
\Lambda_k^{\bar n}(t)\longrightarrow\Lambda_k(t),
\ \text{and}
\left\|
\Gamma_k^{\bar n}(t,\cdot)
-
\Gamma_k(t,\cdot)
\right\|_{L^2(U_0,\nu)}
\longrightarrow0
$$
for $d\mathbb P\otimes dt$-a.e. $(\omega,t)$. After passing
to a further subsequence if necessary, we may also assume
that
$
\Gamma_k^{\bar n}(t,z)
\longrightarrow
\Gamma_k(t,z),
$
$d\mathbb P\otimes dt\otimes\nu(dz)$-a.e.

Since each $P_k^{\bar n}$ is c\`adl\`ag, then
$
P_k^{\bar n}(t-)=P_k^{\bar n}(t)
$
for $d\mathbb P\otimes dt$-a.e. $(\omega,t)$. Hence, after
removing a common null set, the intermediate post-jump bound
\eqref{eq:intermediate-postjump-bound} may be written as
$
\varepsilon_0
\le
P_k^{\bar n}(t)
+
\Gamma_k^{\bar n}(t,z)
\le
K.
$
Passing to the limit gives
\begin{equation}\label{eq:limiting-postjump-bound-pre}
\varepsilon_0
\le
P_k(t)+\Gamma_k(t,z)
\le
K,\ 
d\mathbb P\otimes dt\otimes\nu(dz)\text{-a.e.},\qquad k=1,2.
\end{equation}

Together with \eqref{eq:second-P-limit},
\eqref{eq:limiting-postjump-bound-pre} shows that the
limiting parameters satisfy the same uniform pathwise and
post-jump bounds as the approximating parameters.
Consequently, the selector estimates in Proposition
\ref{prop:hamiltonian-saddle} apply to both the approximate
and limiting Hamiltonians:
\[
|\widehat v_{1,k}^{\,\bar n}|
+
|\widehat v_{2,k}^{\,\bar n}|
\le
C\bigl(1+|\Lambda_k^{\bar n}|\bigr) \ \text{and}\ 
|\widehat v_{1,k}|
+
|\widehat v_{2,k}|
\le
C\bigl(1+|\Lambda_k|\bigr).
\]

At every such point, the sequence
$\{\Lambda_k^{\bar n}(t)\}_{\bar n\ge1}$ is bounded. Therefore, the growth estimates for the minimizing and maximizing selectors imply that the approximate and limiting
optimizers are contained in a common finite ball, which may depend on $(\omega,t)$. Moreover, the outer truncation is
eventually inactive along this further subsequence. On every bounded control set
\[
\begin{aligned}
&\sup_{|v_1|+|v_2|\le R}
\left|
H_k
\big(\omega,t,
v_1,v_2,
P_1^{\bar n},P_2^{\bar n},
\Lambda_k^{\bar n},
\Gamma_1^{\bar n},\Gamma_2^{\bar n}
\big)
-
H_k
\big(\omega,t,
v_1,v_2,
P_1,P_2,
\Lambda_k,
\Gamma_1,\Gamma_2
\big)
\right|
\\
&\qquad\le
C_R
\left[
\sum_{i=1}^2|P_i^{\bar n}-P_i|
+
|\Lambda_k^{\bar n}-\Lambda_k|
+
\sum_{i=1}^2
\|\Gamma_i^{\bar n}-\Gamma_i\|_{L^1(U_0,\nu)}
\right].
\end{aligned}
\]
Since $\nu(U_0)<\infty$, we have 
$\|\Gamma_i^{\bar n}-\Gamma_i\|_{L^1(U_0,\nu)}
\le
\nu(U_0)^{1/2}
\|\Gamma_i^{\bar n}-\Gamma_i\|_{L^2(U_0,\nu)}.
$
It follows that, along the chosen further subsequence,
\begin{align*}
&
\underline H_k^*
\big(\omega,t,
P_1^{\bar n},
P_2^{\bar n},
\Lambda_k^{\bar n},
\Gamma_1^{\bar n},
\Gamma_2^{\bar n}
\big)
+
\bar H_k^{\bar n}
\big(\omega,t,
P_k^{\bar n},
\Lambda_k^{\bar n}
\big)
\longrightarrow
\underline H_k^*
\big(\omega,t,
P_1,
P_2,
\Lambda_k,
\Gamma_1,
\Gamma_2
\big)
+
\bar H_k^*
\big(\omega,t,
P_k,
\Lambda_k
\big),  
\end{align*}
for $ d\mathbb P\otimes dt$-a.e.

Since the original subsequence was arbitrary, the
subsequence principle implies that the whole sequence of
Hamiltonians converges in measure on
$\Omega\times[0,T]$. Moreover,
\[
\left|
\underline H_k^*
\big(\omega,t,
P_1^{\bar n},P_2^{\bar n},
\Lambda_k^{\bar n},
\Gamma_1^{\bar n},\Gamma_2^{\bar n}
\big)
+
\bar H_k^{\bar n}
\big(\omega,t,
P_k^{\bar n},\Lambda_k^{\bar n}
\big)
\right|
\le
C\big(
1+|\Lambda_k^{\bar n}|^2
\big).
\]
By \eqref{eq:second-strong-limit}, we have 
$
|\Lambda_k^{\bar n}|^2
\longrightarrow
|\Lambda_k|^2\ 
\text{in }\ 
L^1(d\mathbb P\otimes dt).
$
Hence,
$\left\{
1+|\Lambda_k^{\bar n}|^2
\right\}_{\bar n\ge1}
$
is uniformly integrable. Therefore, the Hamiltonians are uniformly integrable. Combining their convergence in measure
with Vitali's theorem, we obtain convergence in $L^1(d\mathbb P\otimes dt)$.

Together with \eqref{eq:second-P-limit} and
\eqref{eq:second-strong-limit}, the preceding
$L^1$-convergence of the Hamiltonians implies that the drift
terms in \eqref{eq:intermediate-bsdej} converge in
$L^1(d\mathbb P\otimes dt)$. Moreover, the
BDG inequalities imply that the
Brownian and compensated Poisson stochastic integrals
converge in $\mathcal S_{\mathbb F}^2$. Therefore, passing
to the limit in the integral form of
\eqref{eq:intermediate-bsdej}, we obtain c\`adl\`ag
processes $P_1$ and $P_2$ such that
$
(P_1,P_2,\Lambda_1,\Lambda_2,\Gamma_1,\Gamma_2)
$
satisfies \eqref{csre1}. These c\`adl\`ag processes coincide
with the monotone limits in
$L^2(d\mathbb P\otimes dt)$. Since a c\`adl\`ag process and its left-limit process agree $dt$-a.e.,
\eqref{eq:limiting-postjump-bound-pre} yields
\begin{equation*}\label{eq:limiting-postjump-bound}
\varepsilon_0
\le
P_k(t-)+\Gamma_k(t,z)
\le
K,\ d\mathbb P\otimes dt\otimes\nu(dz)\text{-a.e.},
\qquad k=1,2,
\end{equation*}

Finally, the c\`adl\`ag versions coincide with the monotone
limits $d\mathbb P\otimes dt$-a.e. Therefore,
$
\varepsilon_0\le P_k(t)\le K
$
for $d\mathbb P\otimes dt$-a.e. $(\omega,t)$. By
right-continuity, these inequalities hold simultaneously for
all $t<T$, outside a single $\mathbb P$-null set. At $t=T$,
they follow from
\[
P_k(T)=G(T),
\qquad
\varepsilon_0
\le
\underline\delta
\le
G(T)
\le
\bar c
\le K.
\]
Consequently,
$
\mathbb P\left(
0<P_k(t)\le K,\quad 0\le t\le T
\right)=1,
\qquad k=1,2.
$
\end{proof}

 \end{document}